\documentclass[letterpaper,10pt,conference]{style/ieeeconf}

 \let\labelindent\relax
\IEEEoverridecommandlockouts

\usepackage[hidelinks,bookmarksopen,bookmarksnumbered]{hyperref}

\usepackage[utf8]{inputenc}
\usepackage{xcolor}
\usepackage{booktabs}
\usepackage{graphicx}
\usepackage{times}
\usepackage{amsmath, amsthm}
\usepackage{amssymb}
\usepackage{bm}
\usepackage{microtype}
\usepackage{mathtools}
\usepackage[ruled,algo2e]{algorithm2e}
\usepackage{cite}
\usepackage{adjustbox}
\usepackage{enumitem}
\usepackage{xparse}

\newcommand{\mcal}[1]{\mathcal{#1}}
\newcommand{\T}{\mathsf{T}}

\newcommand{\norm}[1]{\left\Vert #1 \right\Vert}
\newcommand{\mbb}[1]{\mathbb{#1}}

\newcommand{\diag}{\mathrm{diag}}

\newcommand{\minimize}[1]{ \ensuremath{\underset{#1}{\text{minimize}}\ }}
\newcommand{\algsep}{\\[0.4em]}
\newcommand*{\mean}[1]{\ensuremath{\overline{#1}}}
\newcommand*{\trace}[1]{\ensuremath{\mathrm{trace}\left({#1}\right)}}
\newcommand*{\dom}[1]{\ensuremath{\mathrm{dom}\left({#1}\right)}}

\newtheorem{assumption}{Assumption}

\newtheorem{theorem}{Theorem}
\newtheorem{definition}{Definition}
\newtheorem{lemma}{Lemma}
\newtheorem{remark}{Remark}

\def\appendixname{Appendix}

\NewDocumentCommand{\card}{sO{}m}{%
  {\IfBooleanTF{#1}
    {\oldnormaux{\left|\right.}{\left.\right|}{#3}}
    {\oldnormaux{#2|}{#2|}{#3}}}
}
\makeatletter
\newcommand{\oldnormaux}[3]{\mathpalette\oldnormaux@i{{#1}{#2}{#3}}}
\newcommand{\oldnormaux@i}[2]{\oldnormaux@ii#1#2}
\newcommand{\oldnormaux@ii}[4]{%
  \sbox\z@{$\m@th#1#2#4#3$}%
  \sbox\tw@{$\m@th\|$}%
  \mathopen{\hbox to\wd\tw@{\hss\vrule height \ht\z@ depth \dp\z@ width .2\wd\tw@\hss}}%
  #4
  \mathclose{\hbox to\wd\tw@{\hss\vrule height \ht\z@ depth \dp\z@ width .2\wd\tw@\hss}}%
}
\makeatother

\makeatletter
\newcommand{\longdash}[1][2em]{%
  \makebox[#1]{$\m@th\smash-\mkern-7mu\cleaders\hbox{$\mkern-2mu\smash-\mkern-2mu$}\hfill\mkern-7mu\smash-$}}
\makeatother
\newcommand{\omitskip}{\kern-\arraycolsep}
\newcommand{\llongdash}[1][2em]{\longdash[#1]\omitskip}
\newcommand{\rlongdash}[1][2em]{\omitskip\longdash[#1]}

\title{\LARGE \bf Distributed Conjugate Gradient Method via Conjugate Direction Tracking}

\author{Ola Shorinwa$^{1}$ and Mac Schwager$^{2}$%
	\thanks{*This work was supported  in part by NSF NRI awards 1830402 and 1925030 and ONR grant N00014-18-1-2830.}%
	\thanks{$^{1}$Ola Shorinwa is with the Department of Mechanical Engineering, Stanford University, CA, USA {\tt\small shorinwa@stanford.edu}.}%
	\thanks{$^{2}$Mac Schwager is with the Department of Aeronautics and Astronautics Engineering, Stanford University, CA, USA {\tt\small schwager@stanford.edu}.}%
}

\begin{document}
	\maketitle
	\thispagestyle{empty}
	\pagestyle{empty}

	\begin{abstract}
		We present a distributed conjugate gradient method for distributed optimization problems, where each agent computes an optimal solution of the problem locally \emph{without} any central computation or coordination, while communicating with its immediate, one-hop neighbors over a communication network. Each agent updates its local problem variable using an estimate of the \emph{average conjugate direction} across the network, computed via a dynamic consensus approach. Our algorithm enables the agents to use \emph{uncoordinated} step-sizes. We prove convergence of the local variable of each agent to the optimal solution of the aggregate optimization problem, without requiring decreasing step-sizes. In addition, we demonstrate the efficacy of our algorithm in distributed state estimation problems, and its robust counterparts, where we show its performance compared to existing distributed first-order optimization methods.

	\end{abstract}

	\section{Introduction}
\label{sec:introduction}
A variety of problems in many disciplines can be formulated as distributed optimization problems, where a group of agents seek to compute the optimal estimate, action, or control that minimizes (or maximizes) a specified objective function. Examples of such problems include distributed target tracking \cite{zhang2018adaptive, zhu2013distributed, shorinwa2020distributed, park2019distributed}, pose/state/signal estimation in sensor/robotic networks \cite{rabbat2004distributed, necoara2011parallel, mateos2010distributed}; machine learning and statistical modeling \cite{konevcny2016federated, li2020federated, zhou2021communication}; process control \cite{nedic2018distributed, wang2017distributed, erseghe2014distributed};  and multi-agent planning and control \cite{rostami2017admm, tang2019distributed, shorinwa2023distributed}. In these problems, the data is collected and stored locally by each agent, with the additional constraint that no individual agent has access to all the problem data across the network. In many situations, the limited availability of communication and data storage resources, in addition to privacy regulations, preclude the aggregation of the problem data at a central location or node, effectively rendering \emph{centralized optimization} methods infeasible.

\emph{Distributed optimization} enables each agent to compute an optimal solution via local computation procedures while communicating with its neighbors over a communication network. In essence, via distributed optimization, each agent \emph{collaborates} with its neighbors to compute an optimal solution without access to the aggregate problem data.
Some distributed optimization methods require a central coordinator for execution or coordination of some of the update procedures. These methods are often used in machine learning for parallel processing on  a cluster of computing nodes, especially in problems involving large datasets. In contrast, in this work, we focus on \emph{fully-distributed} algorithms that do not require a central node for coordination or computation.

We derive a distributed conjugate gradient algorithm, termed DC-Grad, for distributed optimization problems. In our algorithm, each agent utilizes \emph{first-order} information (i.e., gradients) of its local objective function to compute its local conjugate directions for updating its local estimate of the solution of the optimization problem at each iteration and communicates with its one-hop neighbor over a point-to-point communication network. Each agent does not share its local problem data, including its objective function and gradients, with other agents, preserving the \emph{privacy} of the agents. For simplicity of exposition, we limit our analysis to distributed optimization problems with smooth, convex objective functions. We prove convergence of the local problem variables of all agents to the optimal solution of the aggregate optimization problem. 

We examine the performance of our distributed algorithm in comparison to notable existing distributed optimization methods in distributed state estimation and robust-state-estimation problems. In both problems, we show that our algorithm converges with the least communication overhead in densely-connected communication networks, with some additional computation overhead in comparison to the best-competing distributed algorithm DIGing-ATC. On sparsely-connected graphs, our algorithm performs similarly to other first-order distributed optimization methods.

	 \section{Related Work}
 \label{sec:related_work}
Distributed optimization methods have received significant attention, with many such methods developed from their centralized counterparts. Distributed first-order methods leverage the local gradients (i.e., first-order information) of each agent to iteratively improve the each agent's local estimate of the optimal solution of the optimization problem, bearing similarities with other centralized first-order methods such as the centralized gradient descent. Distributed incremental (sub)gradient methods require a central node that receives the local gradient information from each agent and performs the associated update step \cite{nedic2001distributed}. As such, these methods require a hub-spoke communication model --- where all the agents are connected to the central node (hub) --- or a ring communication model (a cyclic network), which is quite restrictive.

Distributed (sub)gradient methods circumvent this limitation, enabling distributed optimization over arbitrary network topologies. At each iteration, each agent exchanges its local iterates and other auxiliary variables (such as estimates of the average gradient of the joint (global) objective function) with other neighboring  agents. In distributed (sub)gradient methods, each agent recursively updates its local estimate using its local (sub)gradient and \emph{mixes} its estimates with the estimates of its neighbors via a convex combination, where the \emph{mixing} step is achieved via average consensus \cite{nedic2009, matei2011performance} or the push-sum technique \cite{olshevsky2009convergence, benezit2010weighted}. Generally, distributed (sub)gradient methods require a diminishing step-size for convergence to the optimal solution in convex problems \cite{lobel2010distributed, yuan2016convergence}, which typically slows down convergence. With a constant step-size, these methods converge to a neighborhood of the optimal solution.

Distributed gradient-tracking methods were developed to eliminate the need for diminishing step-sizes \cite{shi2015extra, qu2017harnessing, nedic2017achieving}. In these methods, in addition to the local estimate of the optimal solution, each agent maintains an estimate of the \emph{average gradient} of the objective function and updates its local estimate of the optimal solution by taking a descent step in the direction of the estimated average gradient. Distributed gradient-tracking methods provide faster convergence guarantees with constant step-sizes. Further, diffusion-based distributed algorithms \cite{chen2012diffusion, yuan2018exact, yuan2018exact2} converge to the optimal solution with constant step-sizes. Distributed first-order methods have been derived for undirected \cite{shi2015extra, xu2015augmented} and directed \cite{saadatniaki2018optimization, zeng2017extrapush} networks, as well as static \cite{xi2017add, xin2018linear} and time-varying \cite{nedic2017achieving} networks. In addition, acceleration schemes such as momentum and Nesterov acceleration have been applied to distributed first-order methods \cite{xin2019distributedNesterov, qu2019accelerated, lu2020nesterov}.

Distributed methods that leverage higher-order information have been developed for distributed optimization, including distributed quasi-newton methods that approximate the inverse Hessian of the objective function \cite{Mokhtari2015, Eisen2019, mansoori2019fast}. Further, the alternating direction method of multipliers (ADMM) is amenable to consensus optimization problem. In ADMM, each agent maintains an estimate of the optimal solution in addition to dual variables associated with the consensus constraints between agents. However, ADMM, in its original form, requires a central node for computation of the dual update procedure. Fully-distributed variants of ADMM have been developed, addressing this limitation \cite{mateos2010distributed, Ling2015, chang2014multi, farina2019distributed}. In general, ADMM methods are amenable to static, undirected communication networks.

The conjugate gradient (CG) method was originally developed for computing the solution of a linear system of equations (i.e., $Ax = b$), where the matrix ${A \in \mathbb{R}^{n \times n}}$ is square, symmetric, and positive-definite \cite{hestenes1952methods, hager2006survey}. More generally, the method applies to strongly-convex quadratic programming problems, where the conjugate gradient method is guaranteed to compute the optimal solution in at most $n$ iterations, in the absence of roundoff errors. The conjugate gradient method has been extended to nonlinear optimization problems (which includes non-quadratic problems) \cite{dai1999nonlinear, yuan2020conjugate}. In general, the conjugate gradient method provides faster convergence compared to gradient descent methods \cite{gilbert1992global, yuan2019global, shewchuk1994introduction}. Variants of the conjugate gradient method for parallel execution on multiple computing nodes (processors) have been developed \cite{ismail2013implementation, chen2004implementing, lanucara1999conjugate, helfenstein2012parallel, engelmann2021essentially}. These methods decompose the data matrix associated with the linear system of equations into individual components assigned to each processor, enabling parallelization of the matrix-vector operations arising in the conjugate gradient method, which constitute the major computational bottleneck in the CG method. However, these methods are only amenable to problems with hub-spoke communication models or all-to-all communication models. Some other distributed CG methods eliminate the need for a hub-spoke communication model \cite{xu2016distributed}, but, however, require a ring communication model, which does not support parallel execution of the update procedures, ultimately degrading the computational speed of the algorithm. The distributed variant \cite{ping2021dcg} allows for more general communication networks. Nonetheless, these  CG methods are limited to solving a linear system of equations and do not consider a more general optimization problem. Few distributed CG methods for nonlinear optimization problems exist. The work in \cite{xu2020distributed} derives a distributed CG method for online optimization problems where mixing of information is achieved using the average consensus scheme. Like distributed (sub)gradient methods, this algorithm requires a diminishing step-size for convergence to the optimal solution, converging to a neighborhood of the optimal solution if a constant step-size is used.

In this paper, we derive a distributed conjugate gradient method for a more general class of optimization problems, including problems with nonlinear objective functions, and prove convergence of the algorithm to the optimal solution in convex problems with Lipschitz-continuous gradients. Moreover, we note that, in our algorithm, each agent can use uncoordinated constant step-sizes.

		\section{Notation and Preliminaries}
\label{sec:preliminaries}
In this paper, we denote the gradient of a function $f$ by $\nabla f$ and $g$, interchangeably. We denote the all-ones vector as ${\bm{1}_{n} \in \mbb{R}^{n}}$. We represent the inner-product of two matrices ${A \in \mbb{R}^{m \times n}}$ and ${B \in \mbb{}R^{m \times n}}$ as ${\langle A, B \rangle = \trace{A^{\T}B}}$. We denote the standard scalar-vector product, matrix-vector product, and matrix-matrix product (composition) as ${A \cdot B}$, depending on the mathematical context. 
For a given matrix ${A \in \mbb{R}^{m \times n}}$, we denote its spectral norm as ${\rho(A) = \norm{A}_{2}}$. Further, we denote its Frobenius norm by $\norm{A}_{F}$. Likewise, we define the mean of a matrix ${B \in \mbb{R}^{N \times n}}$, computed across its rows, as ${\mean{B} = \frac{1}{N}\bm{1}_{N}\bm{1}_{N}^{\T}B \in \mbb{R}^{N \times n}}$, where each row of $\mean{B}$ is the same. In addition, we define the consensus violation between the matrix ${B \in \mbb{R}^{N \times n}}$ and its mean ${\mean{B} \in \mbb{R}^{N \times n}}$ as ${\tilde{B} = B - \mean{B}}$. We denote the domain of a function $f$ as ${\mathrm{dom}(f)}$, the non-negative orthant as $\mbb{R}_{+}$, and the strictly-positive orthant as $\mbb{R}_{++}$.

We introduce the following definitions that will be relevant to our discussion.

\begin{definition}[Conjugacy]
	Two vectors ${a, b \in \mbb{R}^{n}}$ are conjugate with respect to a symmetric positive-definite matrix ${C \in \mbb{R}^{n \times n}}$ if:
	\begin{equation}
		a^{\T}Cb = \langle a, Cb \rangle = \langle Ca, b \rangle = 0.
	\end{equation}
\end{definition}

\begin{definition}[Convex Function]
	A function ${f: \mathbb{R}^{n} \rightarrow \mathbb{R}}$ is convex if for all ${x,y \in \mathrm{dom}(f)}$ and all ${\zeta \in [0,1]}$:
	\begin{equation}
		f(\zeta x + (1 - \zeta)y) \leq \zeta f(x) + (1 - \zeta) f(y),
	\end{equation}
	and the domain of $f$, ${\dom{f} \subseteq \mathbb{R}^{n}}$, is convex.
\end{definition}

\begin{definition}[Smoothness]
	A function ${f: \mathbb{R}^{n} \rightarrow \mathbb{R}}$ is \mbox{$L$-smooth} if it is continuously differentiable over its domain and its gradients are $L$-Lipschitz continuous, i.e.:
	\begin{equation}
		\norm{\nabla{f}(x) - \nabla{f}(y)}_{2} \leq L \norm{x - y}_{2},\ \forall x, y \in \dom{f},
	\end{equation}
	where ${L \in \mbb{R}_{++}}$ is the Lipschitz constant.
\end{definition}

\begin{definition}[Coercive Function]
	A function ${f: \mathbb{R}^{n} \rightarrow \mathbb{R}^{m}}$ is coercive if ${f(x) \rightarrow \infty}$ as ${x \rightarrow \infty}$, for all ${x \in \dom{f}}$.
\end{definition}

We represent the agents as nodes in an undirected, connected communication graph ${\mcal{G} = (\mcal{V}, \mcal{E})}$, where ${\mcal{V} = \{1, \ldots, N\}}$ denotes the set of vertices, representing the agents, and ${\mcal{E} \subset \mcal{V} \times \mcal{V}}$ represents the set of edges. An edge $(i, j)$ exists in $\mcal{E}$ if agents $i$ and $j$ share a communication link. Moreover, we denote the set of neighbors of agent $i$ as $\mcal{N}_{i}$. We associate a \emph{mixing matrix} ${W \in \mbb{R}^{N \times N}}$ with the underlying communication graph. A mixing matrix $W$ is compatible with $\mcal{G}$ if ${w_{ij} = 0}$,~${\forall j \notin \mcal{N}_{i} \cup \{i\}}$,~${\forall i \in \mcal{V}}$. We denote the \emph{degree} of agent $i$ as ${\deg(i) = |\mcal{N}_{i}|}$, representing the number of neighbors of agent $i$, and the \emph{adjacency matrix} associated with $\mcal{G}$ as ${\mcal{A} \in \mbb{R}^{N \times N}}$, where ${\mcal{A}_{ij} = 1}$ if and only if ${j \in \mcal{N}_{i}}$,~${\forall i \in \mcal{V}}$. In addition, we denote the \emph{graph Laplacian} of $\mcal{G}$ as ${L = \diag(\deg(1),\ldots,\deg(N)) - \mcal{A}}$.
In this work, we make the following assumption on the mixing matrix.

\begin{assumption}
	\label{assm:mixing_matrix}
	The mixing matrix $W$ associated with the communication graph $G$ satisfies:
	\begin{enumerate}
		\item \emph{(Double-Stochasticity)} $W \bm{1} = \bm{1}$ and $\bm{1}^{\T} W = \bm{1},$
		\item \emph{(Spectral Property)} $\lambda = \rho(W - \frac{\bm{1}_{N}\bm{1}_{N}^{\T}}{N}) < 1$. \label{assm:mixing_matrix_spectral}
	\end{enumerate}
\end{assumption}

Part \ref{assm:mixing_matrix_spectral} of Assumption \ref{assm:mixing_matrix} specifies that the matrix ${M = W - \frac{\bm{1}_{N}\bm{1}_{N}^{\T}}{N}}$ has a spectral norm less than one. This assumption is necessary and sufficient for consensus, i.e.,
\begin{equation}
	\lim_{k \rightarrow \infty} W^{k} \rightarrow \frac{\bm{1}_{N}\bm{1}_{N}^{\T}}{N}.
\end{equation}

We note that Assumption \ref{assm:mixing_matrix} is not restrictive, in undirected communication networks. We provide common choices for the mixing matrix $W$:
\begin{enumerate}
	\item \emph{Metropolis-Hastings Weights}:
	\begin{equation*}
		w_{ij} = \begin{cases*}
				\frac{1}{\max\{\deg(i), \deg(j)\} + \epsilon}, & if $(i,j) \in \mcal{E},$ \\
				0 & if $(i,j) \notin \mcal{E}$ and $i \neq j,$ \\
				1 - \sum_{r \in \mcal{V}} w_{ir} & if $i = j,$
			\end{cases*}
	\end{equation*}
	where ${\epsilon \in \mbb{R}_{++}}$ denotes a small positive constant, e.g., ${\epsilon = 1}$ \cite{xiao2007distributed}.

	\item \emph{Laplacian-based Weights}:
	\begin{equation*}
		W = I - \frac{L}{\tau},
	\end{equation*}
	where $L$ denotes the Laplacian matrix of $\mcal{G}$, and ${\tau \in \mbb{R}}$ denotes a scaling parameter with ${\tau > \frac{1}{2} \lambda_{\max}(L)}$. One can choose ${\tau = \max_{i \in \mcal{V}} \{\deg(i)\} + \epsilon}$, if computing $\lambda_{\max}(L)$ is infeasible, where ${\epsilon \in \mbb{R}_{++}}$ represents a small positive constant \cite{sayed2014diffusion}.
\end{enumerate}
	\section{Problem Formulation And The Centralized Conjugate Gradient Method}
We consider the distributed optimization problem:
\begin{equation}
	\label{eq:global_prob}
	\minimize{x \in \mathbb{R}^{n}} \frac{1}{N} \sum_{i = 1}^{N} f_{i}(x),
\end{equation}
over $N$ agents, where ${f_{i}: \mbb{R}^{n} \rightarrow \mathbb{R}}$ denotes the local objective function of agent $i$ and ${x \in \mbb{R}}$ denotes the optimization variable. The objective function of the optimization problem \eqref{eq:global_prob} consists of a sum of $N$ local components, making it \emph{separable}, with each component associated with an agent. We assume that agent $i$ only knows its local objective function $f_{i}$ and has no knowledge of the objective function of other agents. 

We begin with a description of the centralized nonlinear conjugate gradient method, before deriving our method in Section ~\ref{sec:distirbuted_alg}. The nonlinear conjugate gradient method (a generalization of the conjugate gradient method to optimization problems beyond quadratic programs) represents an iterative first-order optimization algorithm that utilizes the gradient of the objective function to generate iterates from the recurrence:
\begin{equation}
	\label{eq:cen_conjugate_method}
	x^{(k+1)} = x^{(k)} + \alpha^{(k)} \cdot s^{(k)},
\end{equation}
where ${x^{(k)} \in \mbb{R}^{n}}$ denotes the estimate at iteration $k$, ${\alpha^{(k)} \in \mbb{R}_{+}}$ denotes the step-size at iteration $k$, and ${s^{(k)} \in \mbb{R}^{n}}$ denotes the conjugate direction at iteration $k$. In the nonlinear conjugate direction method, the conjugate direction is initialized as the negative gradient of the objective function at the initial estimate, with ${s^{(0)} = -g^{(0)}}$. Further, the conjugate directions are generated from the recurrence:
\begin{equation}
	s^{(k + 1)} = -g^{(k+1)} + \beta^{(k)} \cdot s^{(k)},
\end{equation}
at iteration $k$, where ${\beta^{(k)} \in \mbb{R}}$ denotes the conjugate gradient update parameter. Different schemes have been developed for updating the conjugate update parameter. Here, we provide a few of the possible schemes:
\begin{itemize}
	\item \emph{Hestenes-Stiefel Scheme} \cite{hestenes1952methods}:
	\begin{equation}
		\label{eq:hestenes_stiefel_CG_parameter}
		\beta_{HS}^{(k)} = \frac{\left(g^{(k+1)}  - g^{(k)}\right)^\T g^{(k+1)}}{\left(g^{(k+1)}  - g^{(k)}\right)^ \T s^{(k)}}
	\end{equation}

	\item \emph{Fletcher-Reeves Scheme} \cite{fletcher1964function}:
	\begin{equation}
		\label{eq:flecther_reeves_CG_parameter}
		\beta_{FR}^{(k)} = \frac{\norm{g^{(k+1)}}_{2}^{2}}{\norm{g^{(k)}}_{2}^{2}}
	\end{equation}

	\item \emph{Polak-Ribi\`{e}re Scheme} \cite{polak1969note, polyak1969conjugate}:
	\begin{equation}
		\label{eq:polak_ribiere_CG_parameter}
		\beta_{PR}^{(k)} = \frac{\left(g^{(k+1)}  - g^{(k)}\right)^\T g^{(k+1)}}{\norm{g^{(k)}}_{2}^{2}}
	\end{equation}
\end{itemize}

We note that the update schemes are equivalent when $f$ is a strongly-convex quadratic function.
Moreover, when $f$ is strongly-convex and quadratic, the search directions $\{s^{(k)}\}_{\forall k}$ are conjugate. As a result, the iterate $x^{(k)}$ converges to the optimal solution in at most $n$ iterations. For non-quadratic problems, the search directions lose conjugacy, and convergence may occur after more than $n$ iterations. In many practical problems, the value of the update parameter $\beta$ is selected via a hybrid scheme, obtained from a combination of the fundamental update schemes, which include the aforementioned ones. Simple hybrid schemes are also used, e.g., ${\beta^{(k)} = \max\{0, \beta_{PR}^{(k)}\}}$.

	\section{Distributed Conjugate Gradient Method}
\label{sec:distirbuted_alg}
In this section, we derive a distributed optimization algorithm based on the nonlinear conjugate method for \eqref{eq:global_prob}. We assign a local copy of $x$ to each agent, representing its local estimate of the solution of the optimization problem, with each agent computing its conjugate directions locally. Agent $i$ maintains the variables: ${x_{i} \in \mbb{R}^{n}}$, ${s_{i} \in \mbb{R}^{n}}$, along with ${\alpha_{i} \in \mbb{R}_{+}}$, and ${\beta_{i} \in \mbb{R}}$. In addition, we denote the gradient of $f_{i}$ at $x_{i}$ by $g_{i}(x_{i})$. %

Before proceeding with the derivation, we introduce the following notation:
\begin{align*}
	\bm{x} &= \begin{bmatrix}
			\llongdash & x_{1}^{\T} & \rlongdash \\
			 & \vdots &  \\
			\llongdash & x_{N}^{\T} & \rlongdash
		\end{bmatrix}, \
	\bm{s} = \begin{bmatrix}
				\llongdash & s_{1}^{\T} & \rlongdash \\
		 & \vdots &  \\
				\llongdash & s_{N}^{\T} & \rlongdash
			\end{bmatrix}, \\
	\bm{g}(\bm{x}) &= \begin{bmatrix}
			\llongdash & \left(\nabla f_{1}(x_{1})\right)^{\T} & \rlongdash \\
			 & \vdots & \\
			\llongdash & \left(\nabla f_{N}(x_{N})\right)^{\T} & \rlongdash
	\end{bmatrix}, 
\end{align*}
${\bm{\alpha} = \diag(\alpha_{1},\ldots,\alpha_{N})}$, and ${\bm{\beta} = \diag(\beta_{1},\ldots,\beta_{N})}$,
where the variables are obtained by stacking the local variables of each agent, with ${\bm{x} \in \mbb{R}^{N \times n}}$, ${\bm{s} \in \mbb{R}^{N \times n}}$, ${\bm{g}(\bm{x}) \in \mbb{R}^{N \times n}}$, and ${\bm{\alpha} \in \mbb{R}^{N}}$. To simplify notation, we denote $\bm{g}(\bm{x}^{(k)})$ by $\bm{g}^{k}$. In addition, we note that all agents achieve \emph{consensus} when all the rows of $\bm{x}$ are the same. Moreover, \emph{optimality} is achieved when ${\bm{1}_{N}^{\T} \bm{g}(\bm{x}) = \bm{0}_{n}^{\T}}$, i.e., the first-order optimality condition is satisfied.
Further, we define the \emph{aggregate} objective function considering the local variables of each agent as:
\begin{equation}
	\bm{f}(\bm{x}) = \frac{1}{N} \sum_{i = 1}^{N} f_{i}(x_{i}).
\end{equation}

To obtain a distributed variant of the centralized conjugate gradient method, one could utilize the average consensus technique to eliminate the need for centralized procedures, yielding the distributed algorithm:
\begin{align}
	\label{eq:vanilla_distributed_method}
	\bm{x}^{(k+1)} &= W \bm{x}^{(k)} + \bm{\alpha}^{(k)} \cdot \bm{s}^{(k)}, \\
	\bm{s}^{(k + 1)} &= -\bm{g}^{(k+1)} + \bm{\beta}^{(k)} \cdot \bm{s}^{(k)},
\end{align}
which simplifies to:
\begin{align}
	\label{eq:vanilla_distributed_method_agent}
	x_{i}^{(k+1)} &= w_{ii} x_{i}^{(k)} + \sum_{j \in \mcal{N}_{i}} w_{ij} x_{j}^{(k)} + \alpha_{i}^{(k)} \cdot s_{i}^{(k)}, \\
 	s_{i}^{(k + 1)} &= -g_{i}^{(k+1)} + \beta_{i}^{(k)} \cdot s_{i}^{(k)},
\end{align}
when expressed with respect to agent $i$, with initialization ${x_{i}^{(0)} \in \mbb{R}^{n}}$, ${s_{i}^{(0)} = - \nabla f_{i}(x_{i}^{(0)})}$, and ${\alpha_{i}^{(0)} \in \mbb{R}_{+}}$.

One can show that the above distributed algorithm does not converge to the optimal solution with a non-diminishing step-size. Here, we provide a simple proof by contradiction showing that the optimal solution $x^{\star}$ is not a fixed point of the distributed algorithm \eqref{eq:vanilla_distributed_method}: Assume that $x^{\star}$ is a fixed point of the algorithm. With this assumption, the first-two terms on the right-hand side of \eqref{eq:vanilla_distributed_method_agent} simplify to $x^{\star}$. Further, the conjugate update parameter $\beta_{i}^{(k - 1)}$ simplifies to zero, where we define the ratio $\frac{0}{0}$ to be zero if the Fletcher-Reeves Scheme is utilized. However, in general, the local conjugate direction of agent $i$, denoted by $s_{i}^{(k)}$, may not be zero, since the critical point of the joint objective function $\frac{1}{N}\sum_{i = 1}^{N} f_{i}$ may not coincide with the critical point of $f_{i}$, i.e., ${\nabla f_{i}(x^{\star})}$ may not be zero. Consequently, the last term in \eqref{eq:vanilla_distributed_method} is not zero, in general, and as a result, agent $i$'s iterate $x_{i}^{(k+1)}$ deviates from $x^{\star}$, showing that $x^{\star}$ is not a fixed point of the distributed algorithm given by \eqref{eq:vanilla_distributed_method_agent}. This property mirrors that of distributed (sub)gradient methods where a diminishing step-size is required for convergence.

Further, we note that the last term in \eqref{eq:vanilla_distributed_method_agent} is zero if agent $i$ utilizes the average conjugate direction in place of its local conjugate direction. With this modified update procedure, the optimal solution $x^{\star}$ represents a fixed point of the resulting, albeit non-distributed, algorithm. To address this challenge, we assign an auxiliary variable $z$ to each agent, representing an estimate of the average conjugate direction, which is updated locally using dynamic average consensus \cite{zhu2010discrete}, yielding the \emph{Distributed Conjugate Gradient Method} (DC-Grad), given by:
\begin{align}
	\bm{x}^{(k+1)} &= W (\bm{x}^{(k)} + \bm{\alpha}^{(k)} \cdot \bm{z}^{(k)}), \label{eq:x_sequence} \\
	\bm{s}^{(k + 1)} &= -\bm{g}^{(k+1)} + \bm{\beta}^{(k)} \cdot \bm{s}^{(k)}, \label{eq:s_sequence} \\
	\bm{z}^{(k + 1)} &= W(\bm{z}^{(k)} + \bm{s}^{(k+1)} - \bm{s}^{(k)}), \label{eq:z_sequence}
\end{align}
which is initialized with ${x_{i}^{(0)} \in \mbb{R}^{n}}$, ${s_{i}^{(0)} = - \nabla f_{i}(x_{i}^{(0)})}$, ${z_{i}^{(0)} = s_{i}^{(0)}}$, and ${\alpha_{i}^{(0)} \in \mbb{R}_{+}}$,~${\forall i \in \mcal{V}}$.
 Using dynamic average consensus theory, we can show that the agents reach consensus with ${z_{i}^{(\infty)} = \bar{z}^{(\infty)} = \bar{s}^{(\infty)}}$,~${\forall i \in \mcal{V}}$. The resulting distributed conjugate gradient method enables each agent to compute the optimal solution of the optimization problem using \emph{uncoordinated}, \emph{non-diminishing} step-sizes.

Considering the update procedures in terms of each agent, at each iteration $k$, agent $i$ performs the following updates:
\begin{align}
	x_{i}^{(k +1)} &= \sum_{j \in \mcal{N}_{i} \cup \{i\}} w_{ij} \left(x_{j}^{(k)} + \alpha_{j}^{(k)} \cdot z_{j}^{(k)}\right),  \label{eq:x_update_procedure}\\
	s_{i}^{(k +1)} &= -g_{i}^{(k + 1)} + \beta_{i}^{(k)} \cdot s_{i}^{(k)}, \label{eq:s_update_procedure}\\
	z_{i}^{(k + 1)} &= \sum_{j \in \mcal{N}_{i} \cup \{i\}} w_{ij} \left(z_{j}^{(k)} + s_{j}^{(k + 1)} - s_{j}^{(k)} \right), \label{eq:z_update_procedure}
\end{align}
 where agent $i$ communicates:
 \begin{align}
 	u_{i}^{(k)} &= x_{i}^{(k)} + \alpha_{i}^{(k)} \cdot z_{i}^{(k)}, \\
 	v_{i}^{(k)} &= z_{i}^{(k)} + s_{i}^{(k + 1)} - s_{i}^{(k)},
 \end{align}
with its neighbors.

We summarize the distributed conjugate gradient algorithm in Algorithm~\ref{alg:distributed_algorithm}.

\begin{algorithm2e} [th]
	\label{alg:distributed_algorithm}
	\caption{Distributed Conjugate Gradient Method (DC-Grad)}

	\SetKwRepeat{doparallel}{do in parallel}{while}

	\textbf{Initialization:} \\
	{\addtolength{\leftskip}{1em}
		${x_{i}^{(0)} \in \mbb{R}^{n}}$, ${s_{i}^{(0)} = - \nabla f_{i}(x_{i}^{(0)})}$, ${z_{i}^{(0)} = s_{i}^{(0)}}$, and ${\alpha_{i}^{(0)} \in \mbb{R}_{+}}$,~${\forall i \in \mcal{V}}$.
	}

	\doparallel( $\forall i \in \mcal{V}$){not converged or stopping criterion is not met}{
		$x_{i}^{(k + 1)} \leftarrow $ Procedure \eqref{eq:x_update_procedure} \algsep
		$s_{i}^{(k + 1)} \leftarrow $ Procedure \eqref{eq:s_update_procedure} \algsep
		$z_{i}^{(k + 1)} \leftarrow $ Procedure \eqref{eq:z_update_procedure} \algsep
		$k \leftarrow k + 1$
	}

\end{algorithm2e}

We present some assumptions that will be relevant in analyzing the convergence properties of our algorithm.
\begin{assumption}
	\label{assm:convex_lipschitz}
	The local objective function of each agent, $f_{i}$, is closed, proper, and convex. Moreover, $f_{i}$ is $L_{i}$-Lipschitz-smooth, with Lipschitz-continuous gradients.
\end{assumption}

\begin{remark}
	From Assumption \ref{assm:convex_lipschitz}, we note that the aggregate objective function $\bm{f}$ is closed, convex, proper, and Lipschitz-continuous with:
	\begin{equation}
		\norm{\nabla{\bm{f}}(x) - \nabla{\bm{f}}(y)}_{2} \leq L \norm{x - y}_{2},\ \forall x, y \in \mbb{R}^{n},
	\end{equation}
	where ${L = \max_{i \in \mcal{V}} \{L_{i}\}}$.
\end{remark}

\begin{assumption}
	The local objective function of each agent $f_{i}$ is coercive.
\end{assumption}

\begin{assumption}
	The optimization problem \eqref{eq:global_prob} has a non-empty feasible set, and further, an optimal solution $x^{\star}$ exists for the optimization problem.
\end{assumption}

The aforementioned assumptions are standard in convergence analysis of distributed optimization algorithms.
	\section{Convergence Analysis}
\label{sec:convergence_analysis}

We analyze the convergence properties of our distributed algorithm.
Before proceeding with the analysis, we consider the following sequence:
\begin{align}
	\mean{\bm{x}}^{(k+1)} &= M W \bm{x}^{(k+1)} + M \bm{\alpha}^{(k)} \cdot \bm{z}^{(k)}, \\
	\mean{\bm{x}}^{(k+1)} &= \mean{\bm{x}}^{(k)} + \mean{\bm{\alpha}^{(k)} \cdot \bm{z}^{(k)}}, \label{eq:mean_x_sequence} \\
	\mean{\bm{s}}^{(k+1)} &= M \left(-\bm{g}^{(k+1)} + \bm{\beta}^{(k)} \cdot \bm{s}^{(k)}\right), \\
	\mean{\bm{s}}^{(k+1)} &= -\mean{\bm{g}}^{(k+1)} + \mean{\bm{\beta}^{(k)} \cdot \bm{s}^{(k)}}, \label{eq:mean_s_sequence} \\
	\mean{\bm{z}}^{(k+1)} &= M \left(W(\bm{z}^{(k)} + \bm{s}^{(k+1)} - \bm{s}^{(k)})\right), \\
	\mean{\bm{z}}^{(k+1)} &= \mean{\bm{z}}^{(k)} +  \mean{\bm{s}}^{(k+1)} -  \mean{\bm{s}}^{(k)}, \label{eq:mean_z_sequence}
\end{align}
derived from the mean of the local iterates of each agent, where we have utilized the assumption that $W$ is column-stochastic. From \eqref{eq:mean_z_sequence} , we note that ${\mean{z}^{(k)} = \mean{s}^{(k)}}$,~$\forall k$, given that ${\mean{z}^{(0)} = \mean{s}^{(0)}}$.

In addition, we introduce the following definitions: ${\alpha_{\max} = \max_{k \in \mbb{Z}_{+}} \{\norm{\bm{\alpha}^{(k)}}_{2}\}}$; ${\beta_{\max} = \max_{k \in \mbb{Z}_{+}} \{\norm{\bm{\beta}^{(k)}}_{2}\}}$; ${r_{ \alpha} = \alpha_{\max} \max_{k \in \mbb{Z}_{+}} \frac{ 1}{\norm{\mean{\bm{\alpha}}^{(k)}}_{2}}}$; ${r_{\beta} = \beta_{\max} \max_{k \in \mbb{Z}_{+}} \frac{ 1}{\norm{\mean{\bm{\beta}}^{(k)}}_{2}}}$.
Likewise, we define ${\mean{\bm{\alpha}}^{(k)} =   \frac{1}{N} \sum_{i \in \mcal{V}} \alpha_{i}^{(k)} I_{N}}$, with a similar definition for ${\mean{\bm{\beta}}^{(k)}}$. We state the following lemma, bounding the norm of the sequences ${\{\tilde{\bm{x}}^{(k)}\}_{\forall k}}$, ${\{\tilde{\bm{z}}^{(k)}\}_{\forall k}}$, and ${\{\tilde{\bm{s}}^{(k)}\}_{\forall k}}$.

\begin{lemma}
	\label{lem:sequence_bounds}
	If the sequences ${\{\bm{x}^{(k)}\}_{\forall k}}$, ${\{\bm{s}^{(k)}\}_{\forall k}}$, and ${\{\bm{z}^{(k)}\}_{\forall k}}$ are generated by the recurrence in \eqref{eq:x_sequence}, \eqref{eq:s_sequence}, and \eqref{eq:z_sequence}, the auxiliary sequences ${\{\tilde{\bm{x}}^{(k)}\}_{\forall k}}$, ${\{\tilde{\bm{s}}^{(k)}\}_{\forall k}}$, and ${\{\tilde{\bm{z}}^{(k)}\}_{\forall k}}$ satisfy the following bounds:
	\begin{align}
		\norm{\tilde{\bm{x}}^{(k+1)}}_{2} & \leq \lambda \norm{\tilde{\bm{x}}^{(k)}}_{2} + \lambda  \alpha_{\max} (1 +r_{\alpha}) \norm{\tilde{\bm{z}}^{(k)}}_{2} \notag \\
		& \quad + \lambda r_{\alpha} \norm{\mean{\bm{\alpha}^{(k)} \cdot \bm{z}^{(k)}}}_{2}, \\
		\norm{\tilde{\bm{s}}^{(k+1)}}_{2} &\leq \norm{\tilde{\bm{g}}^{(k + 1)}}_{2} + \beta_{\max} (1 + r_{\beta}) \norm{\tilde{\bm{s}}^{(k)}}_{2} \notag \\
		& \quad + \left( 1 + r_{\beta} \right) \norm{\mean{\bm{\beta}^{(k)} \cdot \bm{s}^{(k)}}}_{2}, \\
		\norm{\tilde{\bm{z}}^{(k+1)}}_{2} &\leq (\lambda + \lambda^{2} L \alpha_{\max} ( 1 + r_{\alpha})) \norm{\tilde{\bm{z}}^{(k)}}_{2} \notag \\
		& \quad +  \lambda L ( \lambda  + 1) \norm{\tilde{\bm{x}}^{(k)}}_{2} \notag \\
		& \quad + \lambda  L (\lambda r_{\alpha} + 1)  \norm{\mean{\bm{\alpha}^{(k)} \cdot \bm{z}^{(k)}}}_{2}  \notag \\
		& \quad + \lambda\left( \norm{\bm{\beta}^{(k)} \cdot \bm{s}^{(k)}}_{2} + \norm{\bm{\beta}^{(k - 1)} \cdot \bm{s}^{(k - 1)}}_{2} \right).
	\end{align}
\end{lemma}

\begin{proof}
        \renewcommand{\thesubsection}{\Alph{subsection}}
	Please refer to \hyperref[appdx:lemma_sequence_bounds]{Appendix~\ref*{appdx:lemma_sequence_bounds}} for the proof.
\end{proof}

Further, all agents reach agreement on their local iterates, which we state in the following theorem.

\begin{theorem}[Agreement]
	\label{thm:agreement}
	Given the recurrence \eqref{eq:x_sequence}, \eqref{eq:s_sequence}, and \eqref{eq:z_sequence}, the local iterates of agent $i$, ${\left(x_{i}^{(k)}, s_{i}^{(k)}, z_{i}^{(k)}\right)}$, converge to the mean, ${\forall i \in \mcal{V}}$, i.e., each agent reaches agreement with all other agents, for sufficiently large $k$. In particular:
	\begin{align}
		\lim_{k \rightarrow \infty} \norm{\tilde{\bm{s}}^{(k)}}_{2} = 0, \
		\lim_{k \rightarrow \infty} \norm{\tilde{\bm{x}}^{(k)}}_{2} = 0, \
		\lim_{k \rightarrow \infty} \norm{\tilde{\bm{z}}^{(k)}}_{2} = 0. \label{eq:tilde_errors_theorem}
	\end{align}
	Further, the local iterates of each agent converge to a limit point, as ${k \rightarrow \infty}$, with:
	\begin{align}
		\lim_{k \rightarrow \infty} \norm{\mean{\bm{\alpha}^{(k)} \cdot \bm{z}^{(k)}}}_{2}
		= \lim_{k \rightarrow \infty} \norm{\mean{\bm{\beta}^{(k)} \cdot \bm{s}^{(k)}}}_{2}
		= 0.
	\end{align}
	Moreover, the norm of the mean of the agents' local iterates tracking the average conjugate direction converges to zero, with the norm of the average gradient evaluated at the local iterate of each agent also converging to zero, for sufficiently large $k$. Specifically, the following holds:
	\begin{align}
		\lim_{k \rightarrow \infty} \norm{\mean{\bm{s}}^{(k)}}_{2} = 0, \
		\lim_{k \rightarrow \infty} \norm{\mean{\bm{z}}^{(k)}}_{2} = 0, \
		\lim_{k \rightarrow \infty} \norm{\mean{\bm{g}}^{(k)}}_{2} = 0. \label{eq:mean_gradient_direction_theorem}
	\end{align}
\end{theorem}

\begin{proof}
	We refer readers to \hyperref[appdx:thm_agreement]{Appendix~\ref*{appdx:thm_agreement}} for the proof.
\end{proof}

Theorem \ref{thm:agreement} indicates that the local iterates of all agents, ${\{x_{i}^{(k)}}\}_{\forall i \in \mcal{V}}$, converge to a common limit point $x^{(\infty)}$, given by the mean ${\sum_{i \in \mcal{V}} x_{i}^{(k)}}$, as ${k \rightarrow \infty}$. Further:
\begin{equation}
	\label{eq:mean_gradient_direction_simplified}
	\begin{aligned}
		\lim_{k \rightarrow \infty} \norm{\mean{\bm{g}}^{(k)}}_{2} &= \lim_{k \rightarrow \infty} \norm{\frac{\bm{1}_{N}\bm{1}_{N}^{\T}}{N} \bm{g}^{(k)}}_{2}, \\
		&= \norm{\frac{\bm{1}_{N}\bm{1}_{N}^{\T}}{N} \bm{g}(\bm{x}^{(\infty)})}_{2}, \\
		&= \norm{\bm{1}_{N} \left(\nabla f(x^{(\infty)})\right)^{\T}}_{2}, \\
		&= \norm{\bm{1}_{N}}_{2} \norm{\nabla f(x^{(\infty)})}_{2}, \\
		&= \sqrt{N} \norm{\nabla f(x^{(\infty)})}_{2},
	\end{aligned}
\end{equation}
where ${\bm{x}^{(\infty)} = \bm{1}_{N} \left(x^{(\infty)}\right)^{\T}}$. From \eqref{eq:mean_gradient_direction_theorem} and \eqref{eq:mean_gradient_direction_simplified}, we note that ${\norm{\nabla f(x^{(\infty)})}_{2} = 0}$. Hence, the limit point of the distributed algorithm represents a critical point of the optimization problem \eqref{eq:global_prob}.

\begin{theorem}[Convergence of the Objective Value]
	\label{thm:convergence}
	The value of the objective function $\bm{f}$ evaluated at the mean of the local iterates of all agents converges to the optimal objective value. Moreover, the value of $\bm{f}$ evaluated at the agents' local iterates converges to the optimal objective value, for sufficiently large $k$. Particularly:
	\begin{equation}
	\lim_{k \rightarrow \infty} \bm{f}({\bm{x}}^{(k)}) = \lim_{k \rightarrow \infty} \bm{f}(\mean{\bm{x}}^{(k)}) = f^{\star},
	\end{equation}
\end{theorem}

\begin{proof}
	We provide the proof in \hyperref[appdx:thm_convergence]{Appendix~\ref*{appdx:thm_convergence}}.
\end{proof}

	\section{Simulations}
\label{sec:simulations}
In this section, we examine the performance of our distributed conjugate gradient method (\mbox{DC-GRAD}) in comparison to other existing distributed optimization algorithms, namely: \mbox{DIGing-ATC} \cite{nedic2017achieving}, C-ADMM \cite{mateos2010distributed}, $AB$/Push-Pull \cite{xin2020general}, and $ABm$ \cite{xin2019distributed}, which utilizes \emph{momentum} acceleration to achieve faster convergence. We note that $AB$/Push-Pull reduces to DIGing-CTA when the matrices $A$ and $B$ are selected to be doubly-stochastic. We assess the convergence rate of our algorithm across a range of communication networks, with varying degrees of connectivity, described by the connectivity
ratio ${\kappa = \frac{2 \vert \mcal{E} \vert}{N(N - 1)}}$. We consider a \emph{state estimation problem}, formulated as a least-squares optimization problem, in addition to its robust variant derived with the Huber loss function. In each problem, we utilize \emph{Metropolis-Hastings} weights for the mixing matrix $W$. Since Metropolis-Hastings weights yield doubly-stochastic (DS) mixing matrices, we use the terms $ABm$ and $ABm$-DS interchangeably. We compute the convergence error of the local iterate of each agent to the optimal solution, in terms of the \emph{relative-squared error} (RSE) given by:
\begin{equation}
	\mathrm{RSE} = \frac{\norm{x_{i} - x^{\star}}_{2}}{\norm{x^{\star}}_{2}},
\end{equation}
where $x_{i}$ denotes the local iterate of agent $i$ and $x^{\star}$ denotes the optimal solution, computed from the aggregate optimization problem. We set the threshold for convergence at $1e^{-13}$. For a good comparison of the computation and communication overhead incurred by each method, we selected a convergence threshold that could be attained by all methods. In our simulation study, we note that \mbox{DIGing-ATC} and \mbox{DC-GRAD} yield higher-accuracy solutions compared to the other methods, with $AB$/Push-Pull yielding solutions with the least accuracy.

We utilize the \emph{golden-section search} to select an optimal step-size for our distributed conjugate gradient method, \mbox{DIGing-ATC}, $AB$/Push-Pull, and $ABm$. Likewise, we select an optimal value for the penalty parameter $\rho$ in C-ADMM using golden-section search. Further, we assume that each scalar component in the agents' iterates is represented using the double-precision floating-point representation format.

\subsection{Distributed State Estimation}
\label{sec:dis_state_estimation}
In the state estimation problem, we seek to compute an estimate of a parameter (\emph{state}) given a set of observations (\emph{measurements}). In many situations (e.g., in robotics, process control, and finance), the observations are collected by a network of sensors, resulting in decentralization of the problem data, giving rise to the distributed state estimation problem.
Here, we consider the distributed state estimation problem over a network of $N$ agents, where the agents estimate the state ${x \in \mbb{R}^{n}}$, representing the parameter of interest, such as the location of a target. Each agent makes noisy observations of the state, given by the model: ${y_{i} = C_{i} x + w_{i}}$, where ${y_{i} \in \mbb{R}^{m_{i}}}$ denotes the observations of agent $i$, ${C_{i} \in \mbb{R}^{m_{i} \times n}}$ denotes the observation (measurement) matrix, and $w_{i}$ denotes random noise. We can formulate the state estimation problem as a least-squares optimization problem, given by:
\begin{equation}
	\label{eq:state_estimation_least_squares}
	\minimize{x \in \mbb{R}^{n}} \frac{1}{N} \sum_{i = 1}^{N} \norm{C_{i}x - y_{i}}_{2}^{2}.
\end{equation}

We determine the number of local observations for each agent randomly by sampling from the uniform distribution over the closed interval $[5, 30]$. We randomly generate the problem data: $C_{i}$ and $y_{i}$,~ ${\forall i \in \mcal{V}}$, with ${N = 50}$ and ${n = 10}$. We examine the convergence rate of the distributed optimization algorithms over randomly-generated connected communication graphs. We update the conjugate gradient parameter $\beta$ using a modified \emph{Fletcher-Reeves Scheme} \eqref{eq:flecther_reeves_CG_parameter}.

In Table~\ref{tab:state_estimation_computation_least_squares}, we present the mean and standard deviation of the cumulative computation time per agent, in seconds, required for convergence by each distributed algorithm, over $20$ randomly-generated problems for each communication network. We utilize a closed-form solution for the primal update procedure arising in C-ADMM, making it competitive with other distributed optimization methods in terms of computation time. From Table~\ref{tab:state_estimation_computation_least_squares}, we note that \mbox{DIGing-ATC} requires the shortest computation time, closely followed by \mbox{DC-GRAD}, on densely-connected communication graphs, i.e., on graphs with ${\kappa}$ close to one, where we note that \mbox{DC-GRAD} requires an update procedure for $\beta$, increasing its computation time. However, on more sparsely-connected communication graphs, C-ADMM requires the shortest computation time.

\begin{table*}[th]
	\centering
	\caption{The mean and standard deviation of the cumulative computation time (in seconds) per agent in the distributed state estimation problem.}
	\label{tab:state_estimation_computation_least_squares}
	\begin{adjustbox}{width=\linewidth}
		{\begin{tabular}{c c c c c}
				\toprule
				Algorithm & $\kappa = 0.48$ & $\kappa = 0.80$ & $\kappa = 0.97$ & $\kappa = 1.00$ \\
				\midrule
				$AB$/Push-Pull \cite{xin2020general} & $8.95\mathrm{e}^{-4}  \pm 1.60\mathrm{e}^{-4}$ & $9.42\mathrm{e}^{-4}  \pm 1.44\mathrm{e}^{-4}$ & $1.03\mathrm{e}^{-3}  \pm 1.59\mathrm{e}^{-4}$ & $1.06\mathrm{e}^{-3}  \pm 1.43\mathrm{e}^{-4}$ \\
				$ABm$-DS \cite{xin2019distributed} & $5.54\mathrm{e}^{-4}  \pm 2.23\mathrm{e}^{-5}$ & $3.19\mathrm{e}^{-4}  \pm 3.86\mathrm{e}^{-5}$ & $2.85\mathrm{e}^{-4}  \pm 4.71\mathrm{e}^{-5}$ & $2.91\mathrm{e}^{-4}  \pm 4.10\mathrm{e}^{-5}$ \\
				C-ADMM \cite{mateos2010distributed} & $\bm{1.71\mathrm{e}^{-4}  \pm 1.87\mathrm{e}^{-5}}$ & $\bm{1.34\mathrm{e}^{-4}  \pm 1.20\mathrm{e}^{-5}}$ & $1.18\mathrm{e}^{-4}  \pm 6.84\mathrm{e}^{-6}$ & $1.17\mathrm{e}^{-4}  \pm 7.87\mathrm{e}^{-6}$ \\
				DIGing-ATC \cite{nedic2017achieving} & $5.98\mathrm{e}^{-4}  \pm 3.08\mathrm{e}^{-5}$ & $2.12\mathrm{e}^{-4}  \pm 1.61\mathrm{e}^{-5}$ & $\bm{6.79\mathrm{e}^{-5}  \pm 9.10\mathrm{e}^{-6}}$ & $\bm{3.87\mathrm{e}^{-5}  \pm 2.75\mathrm{e}^{-6}}$ \\
				DC-Grad (ours) & $7.94\mathrm{e}^{-4}  \pm 4.39\mathrm{e}^{-5}$ & $2.85\mathrm{e}^{-4}  \pm 2.04\mathrm{e}^{-5}$ & $9.10\mathrm{e}^{-5}  \pm 9.32\mathrm{e}^{-6}$ & $4.47\mathrm{e}^{-5}  \pm 2.08\mathrm{e}^{-6}$ \\
				\bottomrule
		\end{tabular}}
	\end{adjustbox}
\end{table*}

Moreover, we provide the mean and standard deviation of the cumulative size of messages exchanged per agent, in Megabytes (MB), for each distributed algorithm in Table~\ref{tab:state_estimation_communication_least_squares}. We note that C-ADMM requires agents to communicate fewer variables by a factor of 2, compared to $AB$/Push-Pull, $ABm$, \mbox{DIGing-ATC} , and \mbox{DC-GRAD}. Table~\ref{tab:state_estimation_communication_least_squares} shows that \mbox{DC-GRAD} incurs the least communication overhead for convergence on more-densely-connected graphs, closely followed by \mbox{DIGing-ATC}. This finding reveals that \mbox{DC-GRAD} requires fewer iterations for convergence on these graphs, compared to the other algorithms. On more-sparsely-connected graphs, C-ADMM incurs the least communication overhead.

\begin{table*}[th]
	\centering
	\caption{The mean and standard deviation of the cumulative size of messages exchanged by each agent (in MB) in the distributed state-estimation problem.}
	\label{tab:state_estimation_communication_least_squares}
	\begin{adjustbox}{width=\linewidth}
		{\begin{tabular}{c c c c c}
				\toprule
				Algorithm & $\kappa = 0.48$ & $\kappa = 0.80$ & $\kappa = 0.97$ & $\kappa = 1.00$ \\
				\midrule
				$AB$/Push-Pull \cite{xin2020general} & $6.06\mathrm{e}^{-2}  \pm 1.18\mathrm{e}^{-2}$ & $6.52\mathrm{e}^{-2}  \pm 1.10\mathrm{e}^{-2}$ & $6.97\mathrm{e}^{-2}  \pm 1.23\mathrm{e}^{-2}$ & $7.20\mathrm{e}^{-2}  \pm 1.13\mathrm{e}^{-2}$ \\
				$ABm$-DS \cite{xin2019distributed} & $3.64\mathrm{e}^{-2}  \pm 1.31\mathrm{e}^{-3}$ & $2.13\mathrm{e}^{-2}  \pm 2.21\mathrm{e}^{-3}$ & $1.85\mathrm{e}^{-2}  \pm 3.17\mathrm{e}^{-3}$ & $1.90\mathrm{e}^{-2}  \pm 2.78\mathrm{e}^{-3}$ \\
				C-ADMM \cite{mateos2010distributed} & $\bm{1.15\mathrm{e}^{-2}  \pm 9.80\mathrm{e}^{-4}}$ & $\bm{9.24\mathrm{e}^{-3}  \pm 8.17\mathrm{e}^{-4}}$ & $7.98\mathrm{e}^{-3}  \pm 4.78\mathrm{e}^{-4}$ & $8.03\mathrm{e}^{-3}  \pm 5.00\mathrm{e}^{-4}$ \\
				DIGing-ATC \cite{nedic2017achieving} & $4.68\mathrm{e}^{-2}  \pm 1.27\mathrm{e}^{-3}$ & $1.69\mathrm{e}^{-2}  \pm 1.11\mathrm{e}^{-3}$ & $5.16\mathrm{e}^{-3}  \pm 2.19\mathrm{e}^{-4}$ & $3.00\mathrm{e}^{-3}  \pm 2.20\mathrm{e}^{-4}$ \\
				DC-Grad (ours) & $4.63\mathrm{e}^{-2}  \pm 1.70\mathrm{e}^{-3}$ & $1.70\mathrm{e}^{-2}  \pm 1.10\mathrm{e}^{-3}$ & $\bm{5.16\mathrm{e}^{-3}  \pm 2.00\mathrm{e}^{-4}}$ & $\bm{2.58\mathrm{e}^{-3}  \pm 1.00\mathrm{e}^{-4}}$ \\
				\bottomrule
		\end{tabular}}
	\end{adjustbox}
\end{table*}

In Figure~\ref{fig:algorithm_comparison_fully_connected_least_squares}, we show the convergence error of the agents' iterates, per iteration, on a fully-connected communication network. Figure~\ref{fig:algorithm_comparison_fully_connected_least_squares} highlights that \mbox{DC-GRAD} requires the least number of iterations for convergence, closely followed by \mbox{DIGing-ATC}. In addition, $ABm$ and C-ADMM converge at relatively the same rate. Similarly, we show the convergence error of the iterates of each agent on a randomly-generated connected communication graph with ${\kappa = 0.48}$ in Figure~\ref{fig:algorithm_comparison_non_fully_connected_least_squares_0p48}. We note that C-ADMM converges the fastest in Figure~\ref{fig:algorithm_comparison_non_fully_connected_least_squares_0p48}. In addition, we note that the convergence plot of \mbox{DIGing-ATC} overlays that of \mbox{DC-GRAD}, with both algorithms exhibiting a similar performance.

\begin{figure}[th]
	\centering
	\includegraphics[width=0.85\linewidth]{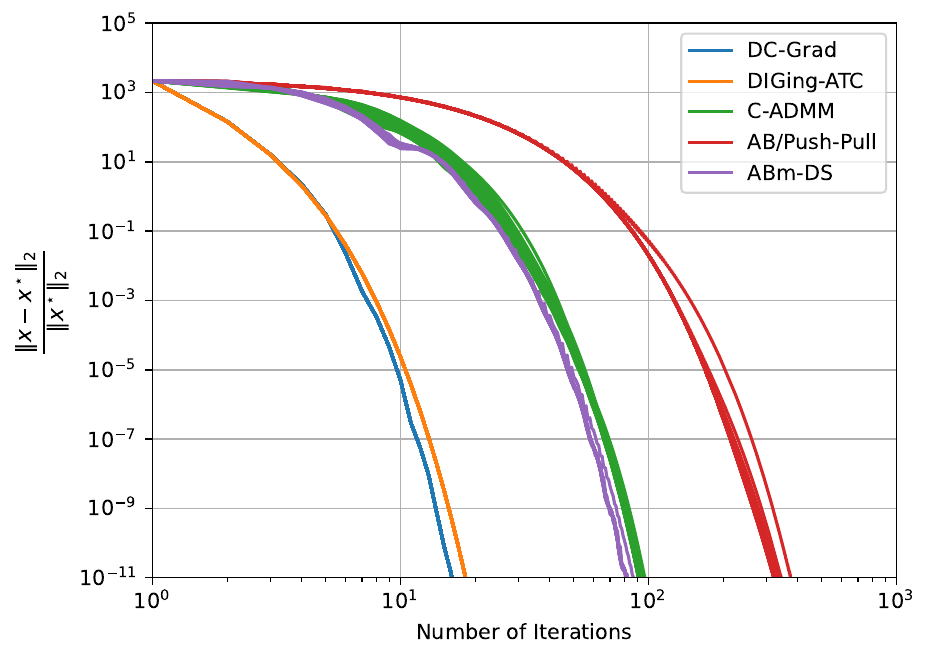}
	\caption{Convergence error of all agents per iteration in the distributed state estimation problem on a fully-connected communication graph. \mbox{DC-GRAD} converges the fastest, closely followed by \mbox{DIGing-ATC}.}
	\label{fig:algorithm_comparison_fully_connected_least_squares}
\end{figure}

\begin{figure}[th]
	\centering
	\includegraphics[width=0.85\linewidth]{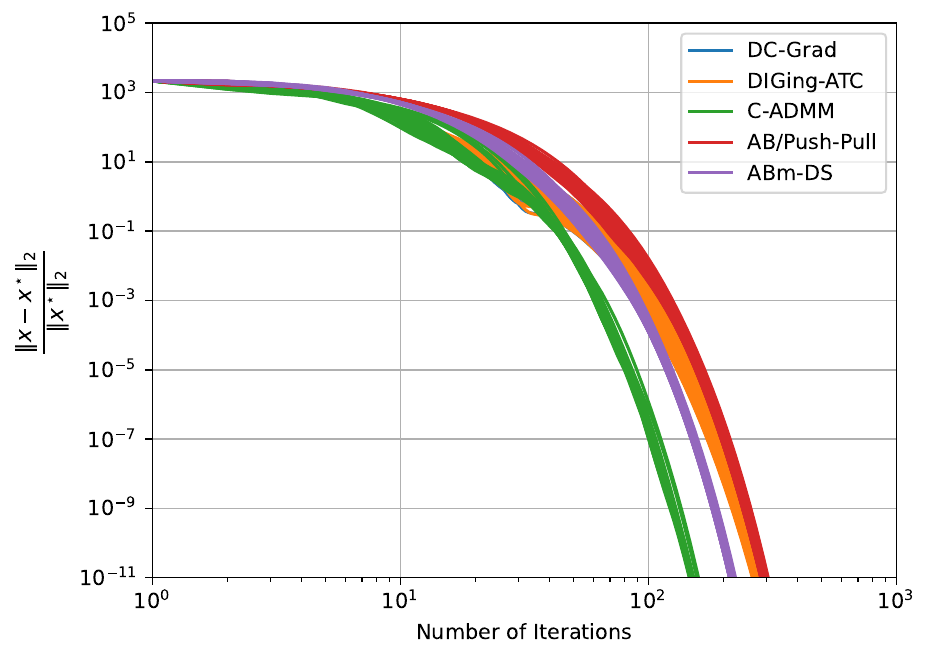}
	\caption{Convergence error of all agents per iteration in the distributed state estimation problem on a randomly-generated connected communication graph with ${\kappa = 0.48}$. C-ADMM attains the fastest convergence rate. The convergence plot of \mbox{DIGing-ATC} \emph{overlays} that of \mbox{DC-GRAD}, with both algorithms converging at the same rate.}
	\label{fig:algorithm_comparison_non_fully_connected_least_squares_0p48}
\end{figure}

\subsection{Distributed Robust Least-Squares}
We consider the robust least-squares formulation of the state estimation problem, presented in Section \ref{sec:dis_state_estimation}. We replace the $\ell_{2}^{2}$-loss function in \eqref{eq:state_estimation_least_squares} with the Huber loss function, given by:
\begin{equation}
	\label{eq:huber_loss}
	f_{\mathrm{hub}, \xi}(u) = \begin{cases}
										\frac{1}{2} u^{2}, & \text{if } \lvert u \rvert \leq \xi \ (\ell_{2}^{2}\text{-zone}), \\
										\xi (\lvert u \rvert - \frac{1}{2}\xi ), & \text{otherwise } (\ell_{1}\text{-zone}).
									\end{cases}
\end{equation}
We note the Huber loss function is less sensitive to outliers, since the penalty function $f_{\mathrm{hub}, \xi}$ grows linearly for large values of $u$. The corresponding robust least-squares optimization problem is given by:
\begin{equation}
	\label{eq:state_estimation_robust_least_squares}
	\minimize{x \in \mbb{R}^{n}} \frac{1}{N} \sum_{i = 1}^{N} f_{\mathrm{hub}, \xi}(C_{i}x - y_{i}).
\end{equation}

We assume each agent has a single observation, i.e., ${m_{i} = 1}$,~${\forall i \in \mcal{V}}$ and assess the convergence rate of the distributed algorithms on randomly-generated connected communication graphs, with ${N = 50}$ and ${n = 10}$. 
We randomly initialize $x_{i}$ such that the $x_{i}$ lies in the $\ell_{1}$-zone,~${\forall i \in \mcal{V}}$. Further, we randomly generate the problem data such that the optimal solution $x^{\star}$ lies in the $\ell_{2}^{2}$-zone. We set the maximum number of iterations to $3000$. We note that a closed-form solution does not exist for the primal update procedure for C-ADMM in this problem. Consequently, we do not include C-ADMM in this study, noting that solving the primal update procedure with iterative solvers would negatively impact the computation time of C-ADMM, effectively limiting its competitiveness. Further, we update the conjugate gradient parameter $\beta$ using a modified \emph{Polak-Ribi\`{e}re Scheme} \eqref{eq:polak_ribiere_CG_parameter}.

We provide the mean computation time per agent, in seconds, required for convergence of each algorithm, along with the standard deviation in Table~\ref{tab:state_estimation_computation_robust_least_squares}, over $20$ randomly-generated problems for each communication network. From Table~\ref{tab:state_estimation_computation_robust_least_squares}, we note that $ABm$ requires the shortest computation time for convergence on more-sparsely-connected communication graphs. However, on more-densely-connected communication graphs, \mbox{DIGing-ATC} achieves the shortest computation time, followed by \mbox{DC-GRAD}.

\begin{table*}[th]
	\centering
	\caption{The mean and standard deviation of the cumulative computation time (in seconds) per agent in the distributed robust-state-estimation problem.}
	\label{tab:state_estimation_computation_robust_least_squares}
	\begin{adjustbox}{width=\linewidth}
		{\begin{tabular}{c c c c c}
				\toprule
				Algorithm & $\kappa = 0.42$ & $\kappa = 0.74$ & $\kappa = 0.96$ & $\kappa = 1.00$ \\
				\midrule
				$AB$/Push-Pull \cite{xin2020general} & $7.55\mathrm{e}^{-3}  \pm 1.89\mathrm{e}^{-3}$ & $7.76\mathrm{e}^{-3}  \pm 1.99\mathrm{e}^{-3}$ & $8.16\mathrm{e}^{-3}  \pm 2.04\mathrm{e}^{-3}$ & $8.63\mathrm{e}^{-3}  \pm 2.02\mathrm{e}^{-3}$ \\
				$ABm$-DS \cite{xin2019distributed} & $8.96\mathrm{e}^{-4}  \pm 1.66\mathrm{e}^{-4}$ & $9.06\mathrm{e}^{-4}  \pm 2.64\mathrm{e}^{-4}$ & $9.47\mathrm{e}^{-4}  \pm 2.96\mathrm{e}^{-4}$ & $9.69\mathrm{e}^{-4}  \pm 2.75\mathrm{e}^{-4}$ \\
				DIGing-ATC \cite{nedic2017achieving} & $\bm{7.43\mathrm{e}^{-4}  \pm 6.89\mathrm{e}^{-5}}$ & $\bm{4.46\mathrm{e}^{-4}  \pm 9.93\mathrm{e}^{-5}}$ & $\bm{1.87\mathrm{e}^{-4}  \pm 4.33\mathrm{e}^{-5}}$ & $\bm{1.71\mathrm{e}^{-4}  \pm 4.51\mathrm{e}^{-5}}$ \\
				DC-Grad (ours) & $1.28\mathrm{e}^{-3}  \pm 1.15\mathrm{e}^{-4}$ & $7.65\mathrm{e}^{-4}  \pm 1.75\mathrm{e}^{-4}$ & $3.05\mathrm{e}^{-4}  \pm 7.27\mathrm{e}^{-5}$ & $2.85\mathrm{e}^{-4}  \pm 6.69\mathrm{e}^{-5}$ \\
				\bottomrule
		\end{tabular}}
	\end{adjustbox}
\end{table*}

In Table~\ref{tab:state_estimation_communication_robust_least_squares}, we show the mean and standard deviation of the cumulative size of messages exchanged by each agent (in MB), in each distributed algorithm. Generally, on more-sparsely-connected graphs, $ABm$ converges the fastest, in terms of the number of iterations, and as a result, incurs the least communication overhead, closely followed by \mbox{DIGing-ATC} and \mbox{DC-GRAD}. On the other hand, on more-densely-connected communication graphs, \mbox{DC-GRAD} incurs the least communication overhead.

\begin{table*}[th]
	\centering
	\caption{The mean and standard deviation of the cumulative size of messages exchanged by each agent (in MB) in the distributed robust-state-estimation problem.}
	\label{tab:state_estimation_communication_robust_least_squares}
	\begin{adjustbox}{width=\linewidth}
		{\begin{tabular}{c c c c c}
				\toprule
				Algorithm & $\kappa = 0.42$ & $\kappa = 0.74$ & $\kappa = 0.96$ & $\kappa = 1.00$ \\
				\midrule
				$AB$/Push-Pull \cite{xin2020general} & $8.86\mathrm{e}^{-1}  \pm 2.26\mathrm{e}^{-1}$ & $9.02\mathrm{e}^{-1}  \pm 2.30\mathrm{e}^{-1}$ & $9.80\mathrm{e}^{-1}  \pm 2.48\mathrm{e}^{-1}$ & $1.01\mathrm{e}^{0}  \pm 2.387\mathrm{e}^{-1}$ \\
				$ABm$-DS \cite{xin2019distributed} & $1.02\mathrm{e}^{-1}  \pm 1.85\mathrm{e}^{-2}$ & $1.02\mathrm{e}^{-1}  \pm 2.99\mathrm{e}^{-2}$ & $1.09\mathrm{e}^{-1}  \pm 3.45\mathrm{e}^{-2}$ & $1.08\mathrm{e}^{-1}  \pm 3.14\mathrm{e}^{-2}$ \\
				DIGing-ATC \cite{nedic2017achieving} & $\bm{1.14\mathrm{e}^{-1}  \pm 1.05\mathrm{e}^{-2}}$ & $\bm{6.65\mathrm{e}^{-2}  \pm 1.48\mathrm{e}^{-2}}$ & $2.81\mathrm{e}^{-2}  \pm 6.74\mathrm{e}^{-3}$ & $2.52\mathrm{e}^{-2}  \pm 6.70\mathrm{e}^{-3}$ \\
				DC-Grad (ours) & $1.14\mathrm{e}^{-1}  \pm 1.06\mathrm{e}^{-2}$ & $6.65\mathrm{e}^{-2}  \pm 1.55\mathrm{e}^{-2}$ & $\bm{2.71\mathrm{e}^{-2}  \pm 6.69\mathrm{e}^{-3}}$ & $\bm{2.47\mathrm{e}^{-2}  \pm 5.95\mathrm{e}^{-3}}$ \\
				\bottomrule
		\end{tabular}}
	\end{adjustbox}
\end{table*}

We show the convergence error of each agent's iterate $x_{i}$, per iteration, on a fully-connected communication network in Figure~\ref{fig:algorithm_comparison_fully_connected_robust_least_squares}. We note that \mbox{DC-GRAD} converges within the fewest number of iterations, closely followed by \mbox{DIGing-ATC}. In addition, $AB$/Push-Pull requires the greatest number of iterations for convergence. We note that $AB$/Push-Pull (which is equivalent to DIGing-CTA) utilizes the \emph{combine-then-adapt} update scheme, which results in slower convergence, generally \cite{nedic2017achieving}. Moreover, the objective function in \eqref{eq:huber_loss} is not strongly-convex over its entire domain, particularly in the $\ell_{1}$-zone. In addition, gradient-tracking methods, in general, require (\emph{restricted}) strong convexity for linear convergence. As a result, all the algorithms exhibit sublinear convergence initially, since all the algorithms are initialized with $x_{i}$ in the $\ell_{1}$-zone,~${\forall i \in \mcal{V}}$. The algorithms exhibit linear convergence when the iterates enter the $\ell_{2}^{2}$-zone, as depicted in Figure~\ref{fig:algorithm_comparison_fully_connected_robust_least_squares}. In addition, Figure~\ref{fig:algorithm_comparison_non_fully_connected_robust_least_squares_0p42} shows the convergence error of each agent's iterates on a randomly-generated communication network with ${\kappa = 0.42}$. On these graphs, $ABm$ requires the least number of iterations for convergence. We note that the convergence plot for \mbox{DIGing-ATC} overlays that of \mbox{DC-GRAD}, with both algorithms exhibiting relatively the same performance.

\begin{figure}[th]
	\centering
	\includegraphics[width=0.85\linewidth]{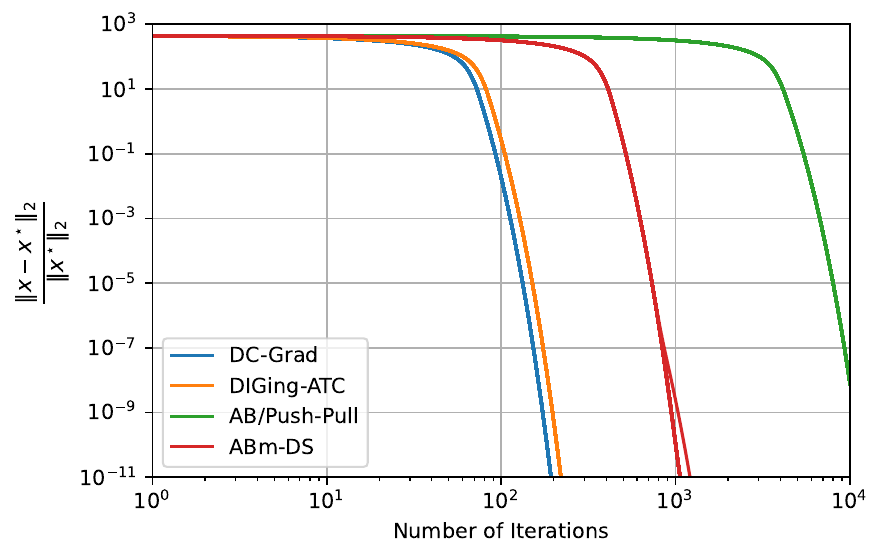}
	\caption{Convergence error of all agents per iteration in the distributed robust-state-estimation problem on a fully-connected communication network. \mbox{DC-GRAD} attains the fastest convergence rate, while $AB$/Push-Pull attains the slowest convergence rate.}
	\label{fig:algorithm_comparison_fully_connected_robust_least_squares}
\end{figure}

\begin{figure}[th]
	\centering
	\includegraphics[width=0.85\linewidth]{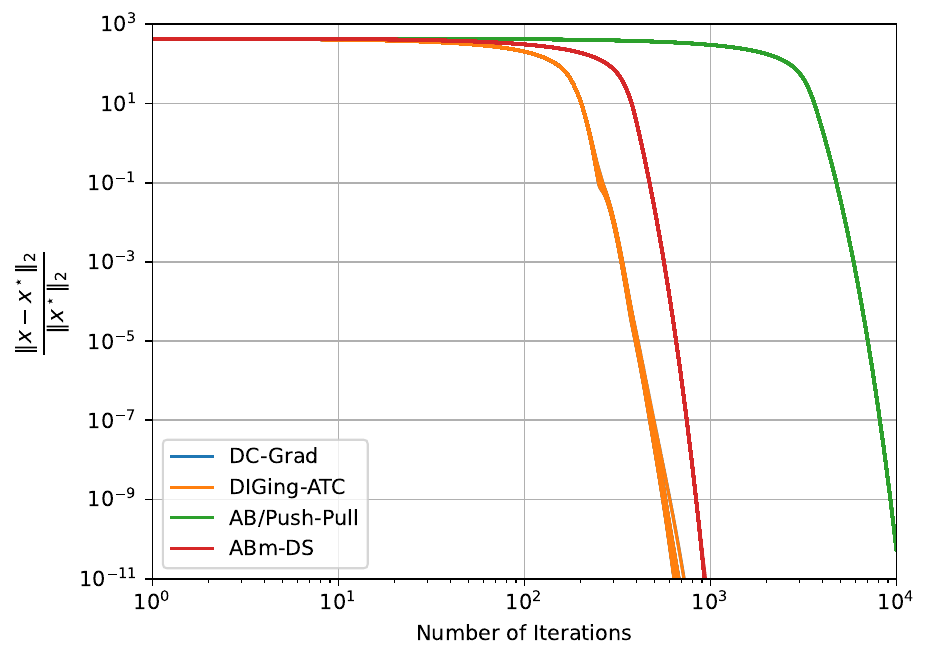}
	\caption{Convergence error of all agents per iteration in the distributed robust-state-estimation problem on a randomly-generated connected communication network with ${\kappa = 0.42}$. The convergence plot of \mbox{DIGing-ATC} \emph{overlays} that of \mbox{DC-GRAD}, with both methods converging faster than the other methods in this trial, although, in general, $ABm$ converges marginally faster on more-sparsely-connected graphs.}
	\label{fig:algorithm_comparison_non_fully_connected_robust_least_squares_0p42}
\end{figure}
	\section{Conclusion}
\label{sec:conclusion}
We introduce DC-Grad, a distributed conjugate gradient method, where each agent communicates with its immediate neighbors to compute an optimal solution of a distributed optimization problem. Our algorithm utilizes only first-order information of the optimization problem, without requiring second-order information. Through simulations, we show that our algorithm requires the least communication overhead for convergence on densely-connected communication graphs, in general, at the expense of a slightly increased computation overhead in comparison to the best-competing algorithm. In addition, on sparsely-connected communication graphs, our algorithms performs similarly to other first-order distributed algorithms. Preliminary convergence analysis of our algorithm suggests that our algorithm converges linearly. In future work, we seek to characterize the convergence rate of our method. Further, in our simulation studies, our algorithm exhibits a notably similar performance with DIGing-ATC. We intend to examine this similarity in future work.
	
\section*{Appendix}
\label{sec:appendix}

\let\oldsubsection = \thesubsection
\renewcommand{\thesubsection}{\Alph{subsection}}

\subsection{Proof of Lemma \ref{lem:sequence_bounds}}
\label{appdx:lemma_sequence_bounds}
Before proceeding with the proof, we state the following relation:
\begin{align}
	\norm{\mean{\bm{\alpha}^{(k)} \cdot \bm{z}^{(k)}}}_{2} &= \norm{\mean{\bm{\alpha}}^{(k)} \cdot \mean{\bm{z}}^{(k)} + \mean{\tilde{\bm{\alpha}}^{(k)} \cdot \tilde{\bm{z}}^{(k)}}}_{2}, \\
	& \geq \norm{\mean{\bm{\alpha}}^{(k)}}_{2} \cdot \norm{\mean{\bm{z}}^{(k)}}_{2} - \norm{\tilde{\bm{\alpha}}^{(k)}}_{2} \cdot \norm{\tilde{\bm{z}}^{(k)}}_{2}, \label{eq:x_update_mean_z_relation}
\end{align}
where we have used the fact that ${\mean{\alpha}^{(k)}  = \norm{\mean{\bm{\alpha}}^{(k)}}_{2}}$ and ${\norm{\bm{1}_{N} \bm{1}_{N}^{\T}}_{2} = \sqrt{N} \sqrt{N} = N}$.

Considering the recurrence in \eqref{eq:mean_z_sequence}:
\begin{equation}
	\label{eq:z_diff_relation}
	\begin{aligned}
		\bm{z}^{(k+1)} - \mean{\bm{z}}^{(k+1)} &= W \bm{z}^{(k)} - \mean{\bm{z}}^{(k)} + W(\bm{s}^{(k+1)} - \bm{s}^{(k)}) \\
		& \quad - (\mean{\bm{s}}^{(k+1)} -  \mean{\bm{s}}^{(k)}), \\
		\tilde{\bm{z}}^{(k+1)} &= M \tilde{\bm{z}}^{(k)} + M (\bm{s}^{(k+1)} - \bm{s}^{(k)}), 
	\end{aligned}
\end{equation}
where we have utilized the relation: ${M \mean{\bm{z}}^{(k)} = 0}$.
From \eqref{eq:z_diff_relation}:
\begin{equation}
	\begin{aligned}
		\norm{\tilde{\bm{z}}^{(k+1)}}_{2} &= \norm{M \tilde{\bm{z}}^{(k)} + M (\bm{s}^{(k+1)} - \bm{s}^{(k)})}_{2}, \\
		&\leq \norm{M \tilde{\bm{z}}^{(k)} }_{2} + \norm{M (\bm{s}^{(k+1)} - \bm{s}^{(k)})}_{2}, \\
		&\leq \lambda \norm{\tilde{\bm{z}}^{(k)}}_{2} + \lambda \norm{\bm{s}^{(k+1)} - \bm{s}^{(k)}}_{2}.
	\end{aligned}
\end{equation}
Using the recurrence \eqref{eq:s_sequence}:
\begin{equation}
	\begin{aligned}
		\norm{\tilde{\bm{z}}^{(k+1)}}_{2} 
		&\leq \lambda \norm{\tilde{\bm{z}}^{(k)}}_{2} +  \lambda \norm{\bm{g}^{(k+1)} - \bm{g}^{(k)}}_{2} \\
		& \quad + \lambda \left(\norm{\bm{\beta}^{(k)} \cdot \bm{s}^{(k)}}_{2} + \norm{\bm{\beta}^{(k - 1)} \cdot \bm{s}^{(k - 1)}}_{2} \right), \\
		&\leq \lambda \norm{\tilde{\bm{z}}^{(k)}}_{2} +  \lambda L  \norm{\bm{x}^{(k+1)} - \bm{x}^{(k)}}_{2} \\
		& \quad + \lambda\left( \norm{\bm{\beta}^{(k)} \cdot \bm{s}^{(k)}}_{2} + \norm{\bm{\beta}^{(k - 1)} \cdot \bm{s}^{(k - 1)}}_{2} \right), 		
	\end{aligned}
\end{equation}
from Lipschitz continuity of $\nabla \bm{f}$, with:
\begin{equation}
	\label{eq:tilde_z_norm_bound_step_a}
	\begin{aligned}
		\norm{\tilde{\bm{z}}^{(k+1)}}_{2} &\leq \lambda \norm{\tilde{\bm{z}}^{(k)}}_{2} \\
		& \quad +  \lambda L  \norm{\tilde{\bm{x}}^{(k+1)} + \mean{\bm{x}}^{(k+1)} - \tilde{\bm{x}}^{(k)} - \mean{\bm{x}}^{(k)} }_{2} \\
		& \quad + \lambda\left( \norm{\bm{\beta}^{(k)} \cdot \bm{s}^{(k)}}_{2} + \norm{\bm{\beta}^{(k - 1)} \cdot \bm{s}^{(k - 1)}}_{2} \right), \\	
		&\leq \lambda \norm{\tilde{\bm{z}}^{(k)}}_{2} \\
		& \quad +  \lambda L \left( \norm{\tilde{\bm{x}}^{(k+1)}}_{2} + \norm{\tilde{\bm{x}}^{(k)}}_{2} \right. \\
        & \hspace{4em} \left. + \norm{\mean{\bm{\alpha}^{(k)} \cdot \bm{z}^{(k)}}}_{2} \right) \\
		& \quad + \lambda\left( \norm{\bm{\beta}^{(k)} \cdot \bm{s}^{(k)}}_{2} + \norm{\bm{\beta}^{(k - 1)} \cdot \bm{s}^{(k - 1)}}_{2} \right), 
	\end{aligned}
\end{equation}
using \eqref{eq:mean_x_sequence}.

In addition: 
\begin{equation}
	\begin{aligned}
		\norm{\tilde{\bm{x}}^{(k+1)}}_{2} &= \norm{\bm{x}^{(k)} - \mean{\bm{x}}^{(k)}}, \\
		&= \norm{W \left(\bm{x}^{(k)} + \bm{\alpha}^{(k)} \cdot \bm{z}^{(k)} \right) - \mean{\bm{x}}^{(k)} - \mean{\bm{\alpha}^{(k)} \cdot \bm{z}^{(k)}}}_{2}, \\
		&= \norm{M \tilde{\bm{x}}^{(k)} + M \left(\bm{\alpha}^{(k)} \cdot \bm{z}^{(k)} - \mean{\bm{\alpha}^{(k)} \cdot \bm{z}^{(k)}} \right)}_{2}, \\
		&= \norm{M \tilde{\bm{x}}^{(k)} + M \left(\bm{\alpha}^{(k)} \cdot \tilde{\bm{z}}^{(k)} + \tilde{\bm{\alpha}}^{(k)} \cdot \mean{\bm{z}}^{(k)} \right)}_{2},
	\end{aligned}
\end{equation}
noting: ${M \mean{\bm{v}} = 0}$. Thus:
\begin{equation}
	\label{eq:tilde_x_norm_bound}
	\begin{aligned}
		\norm{\tilde{\bm{x}}^{(k+1)}}_{2} &\leq \lambda \norm{\tilde{\bm{x}}^{(k)}}_{2} + \lambda \norm{\bm{\alpha}^{(k)} \cdot \tilde{\bm{z}}^{(k)}}_{2} \\
        & \quad + \lambda \norm{\tilde{\bm{\alpha}}^{(k)} \cdot \mean{\bm{z}}^{(k)}}_{2}, \\
		&\leq  \lambda \norm{\tilde{\bm{x}}^{(k)}}_{2} + \lambda \alpha_{\max} \norm{\tilde{\bm{z}}^{(k)}}_{2} \\
		& \quad + \lambda \frac{ \norm{\tilde{\bm{\alpha}}^{(k)}}_{2}}{\norm{\mean{\bm{\alpha}}^{(k)}}_{2}} \cdot \left(\norm{\mean{\bm{\alpha}^{(k)} \cdot \bm{z}^{(k)}}}_{2} \right. \\
		& \hspace{8em} \left. + \norm{\tilde{\bm{\alpha}}^{(k)}}_{2} \cdot \norm{\tilde{\bm{z}}^{(k)}}_{2} \right), \\
		& \leq \lambda \norm{\tilde{\bm{x}}^{(k)}}_{2} + \lambda  \alpha_{\max} (1 +r_{\alpha}) \norm{\tilde{\bm{z}}^{(k)}}_{2} \\
		& \quad + \lambda r_{\alpha} \norm{\mean{\bm{\alpha}^{(k)} \cdot \bm{z}^{(k)}}}_{2}, 
	\end{aligned}
\end{equation}
using \eqref{eq:x_update_mean_z_relation} in the second inequality and the fact ${\norm{\tilde{\bm{\alpha}}^{(k)}}_{2} \leq\alpha_{\max} }$.

Likewise, from \eqref{eq:s_sequence} and \eqref{eq:mean_s_sequence}:
\begin{equation}
	\label{eq:tilde_s_sequence}
	\begin{aligned}
		\tilde{\bm{s}}^{(k+1)} = -\tilde{\bm{g}}^{(k + 1)} + \bm{\beta}^{(k)} \cdot \bm{s}^{(k)} - \mean{\bm{\beta}^{(k)} \cdot \bm{s}^{(k)}}_{2}.
	\end{aligned}
\end{equation}
Considering the second term in \eqref{eq:tilde_s_sequence}:
\begin{equation}
	\begin{aligned}
		\norm{\bm{\beta}^{(k)} \cdot \bm{s}^{(k)}}_{2} &= \norm{\bm{\beta}^{(k)} \cdot \left(\tilde{\bm{s}}^{(k)} + \mean{\bm{s}}^{(k)}\right)}_{2}, \\
		&\leq \beta_{\max} \norm{\tilde{\bm{s}}^{(k)} + \mean{\bm{s}}^{(k)}}_{2}, \\
		&\leq \beta_{\max} \left( \norm{\tilde{\bm{s}}^{(k)}}_{2} + \norm{\mean{\bm{s}}^{(k)}}_{2} \right).
	\end{aligned}
\end{equation}
Further, considering the third term in \eqref{eq:tilde_s_sequence}:
\begin{equation}
	\begin{aligned}
		\norm{\mean{\bm{\beta}^{(k)} \cdot \bm{s}^{(k)}}}_{2} & \geq \norm{\mean{\bm{\beta}}^{(k)}}_{2} \cdot \norm{\mean{\bm{s}}^{(k)}}_{2} - \norm{\tilde{\bm{\beta}}^{(k)}}_{2} \cdot \norm{\tilde{\bm{s}}^{(k)}}_{2};
	\end{aligned}
\end{equation}
hence:
\begin{equation}
	\label{eq:norm_mean_s_upper_bound}
	\norm{\mean{\bm{s}}^{(k)}}_{2} \leq \frac{1}{\norm{\mean{\bm{\beta}}^{(k)}}_{2}} \left( \norm{\mean{\bm{\beta}^{(k)} \cdot \bm{s}^{(k)}}}_{2} + \norm{\tilde{\bm{\beta}}^{(k)}}_{2} \cdot \norm{\tilde{\bm{s}}^{(k)}}_{2} \right),
\end{equation}
which yields:
\begin{equation}
	\begin{aligned}
		\norm{\bm{\beta}^{(k)} \cdot \bm{s}^{(k)}}_{2} &\leq \beta_{\max} (1 + r_{\beta}) \norm{\tilde{\bm{s}}^{(k)}}_{2} \\
	     & \quad + \frac{\beta_{\max}}{\norm{\mean{\bm{\beta}}^{(k)}}_{2}} \norm{\mean{\bm{\beta}^{(k)} \cdot \bm{s}^{(k)}}}_{2}.
	\end{aligned}
\end{equation}
Hence, from \eqref{eq:tilde_s_sequence}:
\begin{equation}
	\label{eq:tilde_s_norm_bound}
	\begin{aligned}
		\norm{\tilde{\bm{s}}^{(k+1)}}_{2} &\leq \norm{\tilde{\bm{g}}^{(k + 1)}}_{2} + \beta_{\max} (1 + r_{\beta}) \norm{\tilde{\bm{s}}^{(k)}}_{2} \\
		& \quad + \left( 1 + \frac{\beta_{\max}}{\norm{\mean{\bm{\beta}}^{(k)}}_{2}}\right) \norm{\mean{\bm{\beta}^{(k)} \cdot \bm{s}^{(k)}}}_{2}, \\
		&\leq \norm{\tilde{\bm{g}}^{(k + 1)}}_{2} + \beta_{\max} (1 + r_{\beta}) \norm{\tilde{\bm{s}}^{(k)}}_{2} \\
		& \quad + \left( 1 + r_{\beta} \right) \norm{\mean{\bm{\beta}^{(k)} \cdot \bm{s}^{(k)}}}_{2}.
	\end{aligned}
\end{equation}
In addition, from \eqref{eq:tilde_z_norm_bound_step_a} and \eqref{eq:tilde_x_norm_bound}:
\begin{equation}
	\label{eq:tilde_z_norm_bound}
	\begin{aligned}
		\norm{\tilde{\bm{z}}^{(k+1)}}_{2} 
		&\leq (\lambda + \lambda^{2} L \alpha_{\max} ( 1 + r_{\alpha})) \norm{\tilde{\bm{z}}^{(k)}}_{2} \\
		& \quad +  \lambda L ( \lambda  + 1) \norm{\tilde{\bm{x}}^{(k)}}_{2} \\
		& \quad + \lambda  L (\lambda r_{\alpha} + 1)  \norm{\mean{\bm{\alpha}^{(k)} \cdot \bm{z}^{(k)}}}_{2}  \\
		& \quad + \lambda\left( \norm{\bm{\beta}^{(k)} \cdot \bm{s}^{(k)}}_{2} + \norm{\bm{\beta}^{(k - 1)} \cdot \bm{s}^{(k - 1)}}_{2} \right).
	\end{aligned}
\end{equation}

\subsection{Proof of Theorem \ref{thm:agreement}}
\label{appdx:thm_agreement}
We introduce the following sequences:
\begin{align}
	X^{(k)} &= \sqrt{\sum_{l = 0}^{k} \norm{\tilde{\bm{x}}^{(l)}}_{2}^{2}}, \quad
	S^{(k)} = \sqrt{\sum_{l = 0}^{k} \norm{\tilde{\bm{s}}^{(l)}}_{2}^{2}}, \\
	Z^{(k)} &= \sqrt{\sum_{l = 0}^{k} \norm{\tilde{\bm{z}}^{(l)}}_{2}^{2}}, \\
	R^{(k)} &= \sqrt{\sum_{l = 0}^{k} \left(\norm{\mean{\bm{\alpha}^{(l)} \cdot \bm{z}^{(l)}}}_{2}^{2}
	+ \norm{\mean{\bm{\beta}^{(l)} \cdot \bm{s}^{(l)}}}_{2}^{2} 
	+ \norm{\tilde{\bm{g}}^{(l + 1)}}_{2}^{2} \right)
	}.
\end{align}

We state the following lemma, and refer readers to \cite{xu2015augmented} for its proof.

\begin{lemma}
	Given the non-negative scalar sequence $\{\nu^{(k)} \}_{\forall k > 0}$, defined by:
	\begin{equation}
		\nu^{(k + 1)} \leq \lambda \nu^{(k)} + \omega^{(k)},
	\end{equation}
	where ${\lambda \in (0, 1)}$, the following relation holds:
	\begin{equation}
	V^{(k + 1)} \leq \gamma \Omega^{(k)} + \epsilon,
	\end{equation}
	where ${V^{(k)} = \sqrt{\sum_{l = 0}^{k} \norm{\nu^{(l)}}_{2}^{2}}}$, ${\Omega^{(k)} = \sqrt{\sum_{l = 0}^{k} \norm{\omega^{(l)}}_{2}^{2}}}$, ${\gamma = \frac{\sqrt{2}}{1 - \lambda}}$, and ${\epsilon =  \nu^{(0)} \sqrt{\frac{2}{1 - \lambda^{2}}}}$.
\end{lemma}

From \eqref{eq:tilde_x_norm_bound}, let:
\begin{equation}
	\omega^{(k)} = \lambda  \alpha_{\max} (1 +r_{\alpha}) \norm{\tilde{\bm{z}}^{(k)}}_{2}
	 + \lambda r_{\alpha} \norm{\mean{\bm{\alpha}^{(k)} \cdot \bm{z}^{(k)}}}_{2},
\end{equation}
which yields:
\begin{equation}
	\label{eq:X_recurrence_step_a}
	\begin{aligned}
		X^{(k)} 
		& \leq \rho_{xz} Z^{(k)} + \rho_{xr} R^{(k)} + \epsilon_{x},
	\end{aligned}
\end{equation}
where ${\rho_{xz} = \frac{\sqrt{2}}{1 - \lambda} \lambda \alpha_{\max} (1 +r_{\alpha})}$, ${\rho_{xr} = \frac{\sqrt{2}}{1 - \lambda} \lambda r_{\alpha}}$, and ${\epsilon_{x} = \norm{\tilde{\bm{x}}^{(0)}}_{2} \sqrt{\frac{2}{1 - \lambda^{2}}}}$.
Likewise, from \eqref{eq:tilde_s_norm_bound}, assuming ${\lambda_{s} = \beta_{\max} (1 + r_{\beta}) < 1}$:
\begin{equation}
	\label{eq:S_recurrence_step}
	\begin{aligned}
		S^{(k)} 
		& \leq \mu_{sr} R^{(k)} + \mu_{sc},
	\end{aligned}
\end{equation}
where ${\mu_{sr} = \rho_{sr} = \frac{\sqrt{2}}{1 - \lambda_{s}} (2 +r_{\beta})}$ and ${\mu_{c} = \epsilon_{s} = \norm{\tilde{\bm{s}}^{(0)}}_{2} \sqrt{\frac{2}{1 - \lambda_{s}^{2}}}}$;
from \eqref{eq:tilde_z_norm_bound}, assuming ${\lambda_{z} = \lambda + \lambda^{2} L \alpha_{\max} ( 1 + r_{\alpha}) < 1}$, 
\begin{equation}
	\label{eq:Z_recurrence_step_a}
	\begin{aligned}
		Z^{(k)} 
		& \leq \rho_{zx} X^{(k)} +  \rho_{zr} R^{(k)} + \epsilon_{z},
	\end{aligned}
\end{equation}
where ${\rho_{zx} = \frac{\sqrt{2}}{1 - \lambda_{z}} \lambda L ( \lambda  + 1)}$ , ${\rho_{zr} = \frac{\sqrt{2}}{1 - \lambda_{z}} \left(\lambda  L (\lambda r_{\alpha} + 1) + 2 \lambda \right)}$ and ${\epsilon_{z} = \norm{\tilde{\bm{z}}^{(0)}}_{2} \sqrt{\frac{2}{1 - \lambda_{z}^{2}}}}$.

From \eqref{eq:X_recurrence_step_a} and \eqref{eq:Z_recurrence_step_a}:
\begin{equation}
	\label{eq:X_recurrence_step}
	\begin{aligned}
		X^{(k)}	
		& \leq \mu_{xr} R^{(k)} + \mu_{xc}, \\
	\end{aligned}
\end{equation}
where ${\mu_{xr} = \frac{\rho_{xz}  \rho_{zr} + \rho_{xr}}{1 - \rho_{xz} \rho_{zx}}}$ and ${\mu_{xc} = \frac{\rho_{xz}  \epsilon_{z} + \epsilon_{x}}{1 - \rho_{xz} \rho_{zx}}}$;
likewise, 
\begin{equation}
	\label{eq:Z_recurrence_step}
	\begin{aligned}
		Z^{(k)}	
		& \leq \mu_{zr} R^{(k)} + \mu_{zc}, \\
	\end{aligned}
\end{equation}
where ${\mu_{zr} = \frac{\rho_{zx} \rho_{xr} +  \rho_{zr}}{\rho_{zx} \rho_{xz}}}$ and ${\mu_{zc} = \frac{\rho_{zx} \epsilon_{x} + \epsilon_{z}}{\rho_{zx} \rho_{xz}}}$.

From $L$-Lipschitz continuity of the gradient:
\begin{equation}
	f_{i}(y) \leq f_{i}(x) + g(x)^{\T}(y - x) + \frac{L_{i}}{2} \norm{y - x}_{2}^{2}, \ \forall i \in \mcal{V}.
\end{equation}

Hence:
\begin{equation}
	\label{eq:L_Lipschitz_upper_bound_step_a}
	\begin{aligned}
		\bm{f}(\mean{\bm{x}}^{(k+1)}) &\leq \bm{f}(\mean{\bm{x}}^{(k)}) \\
            & \quad + \frac{1}{N} \trace{\bm{g}(\mean{\bm{x}}^{(k)}) \cdot (\mean{\bm{x}}^{(k+1)} - \mean{\bm{x}}^{(k)})^{\T} } \\
		& \quad + \frac{1}{N} \cdot \frac{L}{2} \norm{\mean{\bm{x}}^{(k+1)} - \mean{\bm{x}}^{(k)}}_{F}^{2}, \\
		&\leq  \bm{f}(\mean{\bm{x}}^{(k)}) + \frac{1}{N} \norm{\mean{\bm{g}}^{(k)}}_{F} \norm{\mean{\bm{\alpha}^{(k)} \cdot \bm{z}^{(k)}}}_{F} \\
		& \quad + \frac{1}{N} \cdot \frac{L}{2} \norm{\mean{\bm{\alpha}^{(k)} \cdot \bm{z}^{(k)}}}_{F}^{2} \\
		& \quad	+ \frac{1}{N} \norm{\bm{g}(\mean{\bm{x}}^{(k)}) -\bm{g}(\bm{x}^{(k)})}_{F} \norm{\mean{\bm{\alpha}^{(k)} \cdot \bm{z}^{(k)}}}_{F}.
	\end{aligned}
\end{equation}

Considering the second term in \eqref{eq:L_Lipschitz_upper_bound_step_a}:
\begin{equation}
	\label{eq:mean_g_relation_step_a}
	\begin{aligned}
		\norm{\mean{\bm{g}}^{(k)}}_{F} &= \norm{\mean{\bm{\beta}^{(k-1)} \cdot \bm{s}^{(k-1)}} - \mean{\bm{s}}^{(k)}}_{F}, \\
		&\leq \norm{\mean{\bm{\beta}^{(k-1)} \cdot \bm{s}^{(k-1)}}}_{F} + \norm{\mean{\bm{s}}^{(k)}}_{F},
	\end{aligned}
\end{equation}
from \eqref{eq:mean_s_sequence}. 
Further:
\begin{equation}
	\begin{aligned}
		\mean{\bm{\beta}^{(k)} \cdot \bm{s}^{(k)}} &= \mean{\bm{\beta}}^{(k)} \cdot \mean{\bm{s}}^{(k)} + \mean{\tilde{\bm{\beta}}^{(k)} \cdot \tilde{\bm{s}}^{(k)}}, \\
	\end{aligned}
\end{equation}
which shows that:
\begin{equation}
	\label{eq:mean_s_relation}
	\begin{aligned}
		\mean{\bm{s}}^{(k)} = \mean{\bm{\beta}}^{(k)^{\dagger}} \left( \mean{\bm{\beta}^{(k)} \cdot \bm{s}^{(k)}} - \mean{\tilde{\bm{\beta}}^{(k)} \cdot \tilde{\bm{s}}^{(k)}} \right), \\
	\end{aligned}
\end{equation}
where ${\mean{\bm{\beta}}^{(k)^{\dagger}}}$ denotes the inverse of ${\mean{\bm{\beta}}^{(k)}}$.
Hence, from \eqref{eq:mean_g_relation_step_a} and \eqref{eq:mean_s_relation}:
\begin{equation}
	\begin{aligned}
		\norm{\mean{\bm{g}}^{(k)}}_{F} &\leq \norm{\mean{\bm{\beta}^{(k-1)} \cdot \bm{s}^{(k-1)}}}_{F} \\
		& \quad + \frac{1}{\norm{\mean{\bm{\beta}}^{(k)}}_{F}} \left( \norm{\mean{\bm{\beta}^{(k)} \cdot \bm{s}^{(k)}}}_{F} + \norm{\mean{\tilde{\bm{\beta}}^{(k)} \cdot \tilde{\bm{s}}^{(k)}}}_{F} \right), \\
	\end{aligned}
\end{equation}
where ${\norm{\mean{\bm{\beta}}^{(k)^{\dagger}}}_{F} = \frac{\sqrt{N}}{\norm{\mean{\beta}^{(k)}}}_{2}}$.
Further:
\begin{equation}
	\begin{aligned}
		\norm{\mean{\tilde{\bm{\beta}}^{(k)} \cdot \tilde{\bm{s}}^{(k)}}}_{F} &= \norm{\frac{\bm{1}_{N}\bm{1}_{N}^{\T}}{N} \left(\tilde{\bm{\beta}}^{(k)} \cdot \tilde{\bm{s}}^{(k)}\right)}_{F}, \\
		& \leq \norm{\tilde{\bm{\beta}}^{(k)}}_{F} \norm{\tilde{\bm{s}}^{(k)}}_{F}.
	\end{aligned}
\end{equation}
In addition, we note that:
\begin{equation}
	\begin{aligned}
		\norm{\mean{\bm{\beta}^{(k-1)} \cdot \bm{s}^{(k-1)}}}_{F} \norm{\mean{\bm{\alpha}^{(k)} \cdot \bm{z}^{(k)}}}_{F} &\leq \frac{1}{2} \norm{\mean{\bm{\beta}^{(k-1)} \cdot \bm{s}^{(k-1)}}}_{F}^{2} \\
		& \quad + \frac{1}{2} \norm{\mean{\bm{\alpha}^{(k)} \cdot \bm{z}^{(k)}}}_{F}^{2}, 
	\end{aligned}
\end{equation}
from the relation: ${a \cdot b \leq \frac{1}{2} (a^{2} + b^{2})}$.

Hence, from \eqref{eq:L_Lipschitz_upper_bound_step_a}:
\begin{equation}
	\label{eq:L_Lipschitz_upper_bound_step_b}
	\begin{aligned}
		\bm{f}(\mean{\bm{x}}^{(k+1)}) 
		&\leq \bm{f}(\mean{\bm{x}}^{(k)}) \\
		& \quad + \frac{1}{2} \left(\norm{\mean{\bm{\beta}^{(k-1)} \cdot \bm{s}^{(k-1)}}}_{2}^{2}
		+ \norm{\mean{\bm{\alpha}^{(k)} \cdot \bm{z}^{(k)}}}_{2}^{2} \right) \\
		& \quad + \frac{1}{2 \norm{\mean{\bm{\beta}}^{(k)}}_{2}} \left(\norm{\mean{\bm{\beta}^{(k)} \cdot \bm{s}^{(k)}}}_{2}^{2}
		+ \norm{\mean{\bm{\alpha}^{(k)} \cdot \bm{z}^{(k)}}}_{2}^{2} \right) \\
		& \quad + \frac{\norm{\tilde{\bm{\beta}}^{(k)}}_{2}}{\norm{\mean{\bm{\beta}}^{(k)}}_{2}} \norm{\tilde{\bm{s}}^{(k)}}_{2}	\norm{\mean{\bm{\alpha}^{(k)} \cdot \bm{z}^{(k)}}}_{2} \\
		& \quad + \frac{L}{2} \norm{\mean{\bm{\alpha}^{(k)} \cdot \bm{z}^{(k)}}}_{2}^{2} \\
		& \quad	+ L \norm{\tilde{\bm{x}}^{(k)}}_{2} \norm{\mean{\bm{\alpha}^{(k)} \cdot \bm{z}^{(k)}}}_{2},
	\end{aligned}
\end{equation}
where the last term results from Lipschitz continuity of $\nabla f$.
Summing \eqref{eq:L_Lipschitz_upper_bound_step_b} over $k$ from $0$ to $t$:
\begin{equation}
	\begin{aligned}
		\bm{f}(\mean{\bm{x}}^{(t+1)}) 
		&\leq \bm{f}(\mean{\bm{x}}^{(0)}) + \frac{1}{2} \left(1 + \frac{1}{\sqrt{N} \beta_{\max}} +  L \right) \left(R^{(t)}\right)^{2} \\
		& \quad +   r_{\beta} S^{(t)} R^{(t)} +  L X^{(t)} R^{(t)}
	\end{aligned}
\end{equation}
where ${\bm{\beta}^{(-1)} = \bm{s}^{(-1)} = \bm{0}}$, and we have added the term ${\frac{1}{2} \norm{\mean{\bm{\beta}^{(t)} \cdot \bm{s}^{(t)}}}_{2}^{2}}$.

Given \eqref{eq:S_recurrence_step} and \eqref{eq:X_recurrence_step}:
\begin{equation}
	\label{eq:objective_bound}
	\begin{aligned}
		\bm{f}(\mean{\bm{x}}^{(t+1)}) 
		&\leq \bm{f}(\mean{\bm{x}}^{(0)}) + a_{1} \left(R^{(t)}\right)^{2} + a_{2} R^{(t)},
	\end{aligned}
\end{equation}
where ${a_{1} =  \frac{1}{2} \left(1 + \frac{1}{\sqrt{N} \beta_{\max}} + L \right) +  r_{\beta} \mu_{sr} +  L \mu_{xr}}$ and ${a_{2} = r_{\beta} \mu_{sc} + L \mu_{xc}}$.

Subtracting ${f^{\star} = f(x^{\star})}$ from both sides in \eqref{eq:objective_bound} yields:
\begin{equation}
	\begin{aligned}
		\bm{f}(\mean{\bm{x}}^{(t+1)}) - f^{\star} &\leq \bm{f}(\mean{\bm{x}}^{(0)}) - f^{\star}  + a_{1} \left(R^{(t)}\right)^{2} + a_{2} R^{(t)},
	\end{aligned}
\end{equation}
showing that:
\begin{equation}
	0 \leq \bm{f}(\mean{\bm{x}}^{(t+1)}) - f^{\star}  \leq \bm{f}(\mean{\bm{x}}^{(0)}) - f^{\star}  + a_{1} \left(R^{(t)}\right)^{2} + a_{2} R^{(t)}.
\end{equation}
Hence:
\begin{equation}
	a_{1} \left(R^{(t)}\right)^{2} + a_{2} R^{(t)} + \bm{f}(\mean{\bm{x}}^{(0)}) - f^{\star} \geq 0.
\end{equation}

We note that ${a_{1} < 0}$ when:
\begin{equation}
	\frac{1}{2} \left(1 + \frac{1}{\sqrt{N} \beta_{\max}} + L \right) +  r_{\beta} \mu_{sr} + L \mu_{xr} < 0,
\end{equation}
while ${a_{2} \geq 0}$.

With ${a_{1} < 0}$:
\begin{equation}
	-a_{1} \left(R^{(t)}\right)^{2} - a_{2} R^{(t)} - \left(\bm{f}(\mean{\bm{x}}^{(0)}) - f^{\star}\right) \leq 0,
\end{equation}
which yields:
\begin{equation}
	\label{eq:R_Sequence_lim}
	\lim_{t \rightarrow \infty} R^{(t)} \leq \frac{a_{2} + \sqrt{a_{2}^{2} - 4 a_{1} \left(\bm{f}(\mean{\bm{x}}^{(0)}) - f^{\star}\right)}}{- 2 a_{1}} = \Sigma < \infty. 
\end{equation}

Using the monotone convergence theorem in \eqref{eq:R_Sequence_lim}: ${R^{(t)} \rightarrow \Sigma}$, and:
\begin{equation}
	\lim_{k \rightarrow \infty} \left(\norm{\mean{\bm{\alpha}^{(k)} \cdot \bm{z}^{(k)}}}_{2}^{2}
	+ \norm{\mean{\bm{\beta}^{(k)} \cdot \bm{s}^{(k)}}}_{2}^{2} 
	+ \norm{\tilde{\bm{g}}^{(k+ 1)}}_{2}^{2} \right) = 0,
\end{equation}
which shows that:
\begin{equation}
	\begin{aligned}
		\lim_{k \rightarrow \infty} \norm{\mean{\bm{\alpha}^{(k)} \cdot \bm{z}^{(k)}}}_{2}^{2}
		&= \lim_{k \rightarrow \infty} \norm{\mean{\bm{\beta}^{(k)} \cdot \bm{s}^{(k)}}}_{2}^{2} \\
		&= \lim_{k \rightarrow \infty} \norm{\tilde{\bm{g}}^{(k+ 1)}}_{2}^{2} \\
		&= 0.
	\end{aligned}
\end{equation}

From \eqref{eq:S_recurrence_step}:
\begin{equation}
	\lim_{k \rightarrow \infty} S^{(k)} \leq \lim_{k \rightarrow \infty}  (\mu_{sr} R^{(k)} + \mu_{sc}) 
	\leq \mu_{sr} \Sigma + \mu_{sc} < \infty.
\end{equation}
Similarly, from the monotone convergence theorem:
\begin{equation}
	\label{eq:s_tilde_agreement}
	\lim_{k \rightarrow \infty} \norm{\tilde{\bm{s}}^{(k)}}_{2} = 0.
\end{equation}
Likewise, from \eqref{eq:X_recurrence_step}:
\begin{equation}
	\lim_{k \rightarrow \infty} X^{(k)} \leq \lim_{k \rightarrow \infty}  (\mu_{xr} R^{(k)} + \mu_{xc}) 
	\leq \mu_{xr} \Sigma + \mu_{xc} < \infty;
\end{equation}
\begin{equation}
	\label{eq:x_tilde_agreement}
	\lim_{k \rightarrow \infty} \norm{\tilde{\bm{x}}^{(k)}}_{2} = 0.
\end{equation}
Similarly, from \eqref{eq:Z_recurrence_step}:
\begin{equation}
	\lim_{k \rightarrow \infty} Z^{(k)} \leq \lim_{k \rightarrow \infty}  (\mu_{zr} R^{(k)} + \mu_{zc}) 
	\leq \mu_{zr} \Sigma + \mu_{zc} < \infty,
\end{equation}
showing that:
\begin{equation}
	\label{eq:z_tilde_agreement}
	\lim_{k \rightarrow \infty} \norm{\tilde{\bm{z}}^{(k)}}_{2} = 0.
\end{equation}

From \eqref{eq:s_tilde_agreement}, \eqref{eq:x_tilde_agreement}, and \eqref{eq:z_tilde_agreement}, we note that the agents reach \emph{agreement} or \emph{consensus}, with the local iterate of each agent converging to the mean as ${k \rightarrow \infty}$.

Moreover, from \eqref{eq:x_update_mean_z_relation}:
\begin{align}
	\norm{\mean{\bm{z}}^{(k)}}_{2} &\leq \frac{1}{\norm{\mean{\bm{\alpha}}^{(k)}}_{2}} \norm{\mean{\bm{\alpha}^{(k)} \cdot \bm{z}^{(k)}}}_{2} + \frac{\norm{\tilde{\bm{\alpha}}^{(k)}}_{2}}{\norm{\mean{\bm{\alpha}}^{(k)}}_{2}} \cdot \norm{\tilde{\bm{z}}^{(k)}}_{2};
\end{align}
hence:
\begin{equation}
	\begin{aligned}
		\lim_{k \rightarrow \infty} \norm{\mean{\bm{z}}^{(k)}}_{2} &\leq \lim_{k \rightarrow \infty} \left(\frac{1}{\norm{\mean{\bm{\alpha}}^{(k)}}_{2}} \cdot \norm{\mean{\bm{\alpha}^{(k)} \cdot \bm{z}^{(k)}}}_{2} \right. \\
		& \hspace{6em} \left. + \frac{\norm{\tilde{\bm{\alpha}}^{(k)}}_{2}}{\norm{\mean{\bm{\alpha}}^{(k)}}_{2}} \cdot \norm{\tilde{\bm{z}}^{(k)}}_{2}\right), \\
		&= 0,
	\end{aligned}
\end{equation}
yielding:
\begin{equation}
	\lim_{k \rightarrow \infty} \norm{\mean{\bm{z}}^{(k)}}_{2} = 0.
\end{equation}

Likewise, from \eqref{eq:norm_mean_s_upper_bound}:
\begin{equation}
	\begin{aligned}
		\lim_{k \rightarrow \infty} \norm{\mean{\bm{s}}^{(k)}}_{2} &\leq \lim_{k \rightarrow \infty}  \left( \frac{1}{\norm{\mean{\bm{\beta}}^{(k)}}_{2}} \cdot \norm{\mean{\bm{\beta}^{(k)} \cdot \bm{s}^{(k)}}}_{2} \right. \\
		& \hspace{6em} \left. + \frac{\norm{\tilde{\bm{\beta}}^{(k)}}_{2}}{\norm{\mean{\bm{\beta}}^{(k)}}_{2}}  \cdot \norm{\tilde{\bm{s}}^{(k)}}_{2} \right), \\
		&= 0,
	\end{aligned}
\end{equation}
giving the result:
\begin{equation}
	\lim_{k \rightarrow \infty} \norm{\mean{\bm{s}}^{(k)}}_{2} = 0.
\end{equation}

Further, from \eqref{eq:mean_s_sequence}:
\begin{equation}
	\lim_{k \rightarrow \infty} \norm{\mean{\bm{g}}^{(k)}}_{2} \leq \lim_{k \rightarrow \infty}  \left(\norm{\mean{\bm{s}}^{(k)}}_{2} + \norm{\mean{\bm{\beta}^{(k - 1)} \cdot \bm{s}^{(k - 1)}}}_{2} \right)= 0,
\end{equation}
yielding:
\begin{equation}
	\lim_{k \rightarrow \infty} \norm{\mean{\bm{g}}^{(k)}}_{2} = 0.
\end{equation}

\subsection{Proof of Theorem \ref{thm:convergence}}
\label{appdx:thm_convergence}
Since $\bm{f}$ is convex: 
\begin{equation}
	\label{eq:objective_bound_convex}
	\begin{aligned}
		\bm{f}(\mean{\bm{x}}^{(k)}) - f^{\star} &\leq \frac{1}{N} \trace{\bm{g}(\mean{\bm{x}}^{(k)}) \cdot ( \mean{\bm{x}}^{(k)} - \bm{x}^{\star})^{\T} }, \\
		&\leq \frac{1}{N}  \norm{\mean{\bm{g}}(\bm{x}^{(k)})}_{F} \norm{\mean{\bm{x}}^{(k)} - \bm{x}^{\star}}_{F} \\
		& \quad + \frac{1}{N}  \norm{\bm{g}(\mean{\bm{x}}^{(k)}) - \bm{g}({\bm{x}}^{(k)})}_{F} \norm{\mean{\bm{x}}^{(k)} - \bm{x}^{\star}}_{F}, \\
		&\leq  \norm{\mean{\bm{g}}(\bm{x}^{(k)})}_{2} \norm{\mean{\bm{x}}^{(k)} - \bm{x}^{\star}}_{2} \\
		& \quad + \frac{L}{2}\norm{\mean{\bm{x}}^{(k)} - \bm{x}^{(k)}}_{2}  \norm{\mean{\bm{x}}^{(k)} - \bm{x}^{\star}}_{2},
	\end{aligned}
\end{equation}
where ${\bm{x}^{\star} = \bm{1}_{N} \left(x^{\star}\right)^{\T}}$.

Since $\bm{f}$ is coercive by assumption and $\bm{f}(\mean{\bm{x}}^{(k)})$ is bounded from \eqref{eq:objective_bound}, $\norm{\mean{\bm{x}}^{(k)}}_{2}$ is bounded, and thus, ${\norm{\mean{\bm{x}}^{(k)} - \bm{x}^{\star}}_{2} \leq \norm{\mean{\bm{x}}^{(k)}}_{2} + \norm{\bm{x}^{\star}}_{2}}$ is bounded.
Hence:
\begin{equation}
	\lim_{k \rightarrow \infty} \left(\bm{f}(\mean{\bm{x}}^{(k)}) - f^{\star} \right) \leq 0,
\end{equation}
which indicates that:
\begin{equation}
	\label{eq:objective_mean_convergence}
	\lim_{k \rightarrow \infty} \bm{f}(\mean{\bm{x}}^{(k)}) = f^{\star}.
\end{equation}

From the mean-value theorem:
\begin{equation}
	\label{eq:f_bound_mean_value_theorem}
	\bm{f}({\bm{x}}^{(k)}) = \bm{f}(\mean{\bm{x}}^{(k)}) + \frac{1}{N} \trace{\bm{g}(\mean{\bm{x}}^{(k)} + \xi \tilde{\bm{x}}^{(k)}) \cdot \left(\tilde{\bm{x}}^{(k)}\right)^{\T} },
\end{equation}
where ${0 \leq \xi \leq 1}$.

In addition, ${\norm{\mean{\bm{x}}^{(k)} + \xi \tilde{\bm{x}}^{(k)}}_{2} \leq \norm{\mean{\bm{x}}^{(k)}}_{2} + \xi \norm{\tilde{\bm{x}}^{(k)}}_{2}}$ is bounded, as well as ${\norm{\bm{g}(\mean{\bm{x}}^{(k)} + \xi \tilde{\bm{x}}^{(k)})}_{2}}$, since $\bm{g}$ is Lipschitz-continuous. As a result, from \eqref{eq:f_bound_mean_value_theorem},
\begin{equation}
	\begin{aligned}
		&\left\lvert \bm{f}({\bm{x}}^{(k)}) - \bm{f}(\mean{\bm{x}}^{(k)}) \right\rvert \\
        & \quad = \frac{1}{N} \left\lvert \trace{\bm{g}(\mean{\bm{x}}^{(k)} + \xi \tilde{\bm{x}}^{(k)}) \cdot \left(\tilde{\bm{x}}^{(k)}\right)^{\T} } \right\rvert, \\
		& \quad \leq \norm{\bm{g}(\mean{\bm{x}}^{(k)} + \xi \tilde{\bm{x}}^{(k)})}_{2} \norm{\tilde{\bm{x}}^{(k)}}_{2}.
	\end{aligned}
\end{equation}

Hence:
\begin{equation}
	\begin{aligned}
		\lim_{k \rightarrow \infty} \left\lvert \bm{f}({\bm{x}}^{(k)}) - \bm{f}(\mean{\bm{x}}^{(k)}) \right\rvert 
  &\leq \lim_{k \rightarrow \infty} \left( \norm{\bm{g}(\mean{\bm{x}}^{(k)} + \xi \tilde{\bm{x}}^{(k)})}_{2} \right. \\
  & \hspace{4em} \left. \cdot \norm{\tilde{\bm{x}}^{(k)}}_{2} \right), \\
		&= 0,
	\end{aligned}
\end{equation}
from \eqref{eq:x_tilde_agreement}.
As a result:
\begin{equation}
	\lim_{k \rightarrow \infty} \bm{f}({\bm{x}}^{(k)}) = \lim_{k \rightarrow \infty} \bm{f}(\mean{\bm{x}}^{(k)}) = f^{\star},
\end{equation}
from \eqref{eq:objective_mean_convergence}, proving convergence to the optimal objective value.

\let\thesubsection = \oldsubsection

	\bibliographystyle{style/IEEEtran}
	\bibliography{references}

\begin{thebibliography}{10}
\providecommand{\url}[1]{#1}
\csname url@samestyle\endcsname
\providecommand{\newblock}{\relax}
\providecommand{\bibinfo}[2]{#2}
\providecommand{\BIBentrySTDinterwordspacing}{\spaceskip=0pt\relax}
\providecommand{\BIBentryALTinterwordstretchfactor}{4}
\providecommand{\BIBentryALTinterwordspacing}{\spaceskip=\fontdimen2\font plus
\BIBentryALTinterwordstretchfactor\fontdimen3\font minus \fontdimen4\font\relax}
\providecommand{\BIBforeignlanguage}[2]{{%
\expandafter\ifx\csname l@#1\endcsname\relax
\typeout{** WARNING: IEEEtran.bst: No hyphenation pattern has been}%
\typeout{** loaded for the language `#1'. Using the pattern for}%
\typeout{** the default language instead.}%
\else
\language=\csname l@#1\endcsname
\fi
#2}}
\providecommand{\BIBdecl}{\relax}
\BIBdecl

\bibitem{zhang2018adaptive}
H.~Zhang, X.~Zhou, Z.~Wang, H.~Yan, and J.~Sun, ``Adaptive consensus-based distributed target tracking with dynamic cluster in sensor networks,'' \emph{IEEE transactions on cybernetics}, vol.~49, no.~5, pp. 1580--1591, 2018.

\bibitem{zhu2013distributed}
S.~Zhu, C.~Chen, W.~Li, B.~Yang, and X.~Guan, ``Distributed optimal consensus filter for target tracking in heterogeneous sensor networks,'' \emph{IEEE transactions on cybernetics}, vol.~43, no.~6, pp. 1963--1976, 2013.

\bibitem{shorinwa2020distributed}
O.~Shorinwa, J.~Yu, T.~Halsted, A.~Koufos, and M.~Schwager, ``Distributed multi-target tracking for autonomous vehicle fleets,'' in \emph{2020 IEEE International Conference on Robotics and Automation (ICRA)}.\hskip 1em plus 0.5em minus 0.4em\relax IEEE, 2020, pp. 3495--3501.

\bibitem{park2019distributed}
S.-S. Park, Y.~Min, J.-S. Ha, D.-H. Cho, and H.-L. Choi, ``A distributed admm approach to non-myopic path planning for multi-target tracking,'' \emph{IEEE Access}, vol.~7, pp. 163\,589--163\,603, 2019.

\bibitem{rabbat2004distributed}
M.~Rabbat and R.~Nowak, ``Distributed optimization in sensor networks,'' in \emph{Proceedings of the 3rd international symposium on Information processing in sensor networks}, 2004, pp. 20--27.

\bibitem{necoara2011parallel}
I.~Necoara, V.~Nedelcu, and I.~Dumitrache, ``Parallel and distributed optimization methods for estimation and control in networks,'' \emph{Journal of Process Control}, vol.~21, no.~5, pp. 756--766, 2011.

\bibitem{mateos2010distributed}
G.~Mateos, J.~A. Bazerque, and G.~B. Giannakis, ``Distributed sparse linear regression,'' \emph{IEEE Transactions on Signal Processing}, vol.~58, no.~10, pp. 5262--5276, 2010.

\bibitem{konevcny2016federated}
J.~Kone{\v{c}}n{\`y}, H.~B. McMahan, D.~Ramage, and P.~Richt{\'a}rik, ``Federated optimization: Distributed machine learning for on-device intelligence,'' \emph{arXiv preprint arXiv:1610.02527}, 2016.

\bibitem{li2020federated}
T.~Li, A.~K. Sahu, M.~Zaheer, M.~Sanjabi, A.~Talwalkar, and V.~Smith, ``Federated optimization in heterogeneous networks,'' \emph{Proceedings of Machine learning and systems}, vol.~2, pp. 429--450, 2020.

\bibitem{zhou2021communication}
Y.~Zhou, Q.~Ye, and J.~Lv, ``Communication-efficient federated learning with compensated overlap-fedavg,'' \emph{IEEE Transactions on Parallel and Distributed Systems}, vol.~33, no.~1, pp. 192--205, 2021.

\bibitem{nedic2018distributed}
A.~Nedi{\'c} and J.~Liu, ``Distributed optimization for control,'' \emph{Annual Review of Control, Robotics, and Autonomous Systems}, vol.~1, pp. 77--103, 2018.

\bibitem{wang2017distributed}
Y.~Wang, S.~Wang, and L.~Wu, ``Distributed optimization approaches for emerging power systems operation: A review,'' \emph{Electric Power Systems Research}, vol. 144, pp. 127--135, 2017.

\bibitem{erseghe2014distributed}
T.~Erseghe, ``Distributed optimal power flow using admm,'' \emph{IEEE transactions on power systems}, vol.~29, no.~5, pp. 2370--2380, 2014.

\bibitem{rostami2017admm}
R.~Rostami, G.~Costantini, and D.~G{\"o}rges, ``Admm-based distributed model predictive control: Primal and dual approaches,'' in \emph{2017 IEEE 56th Annual Conference on Decision and Control (CDC)}.\hskip 1em plus 0.5em minus 0.4em\relax IEEE, 2017, pp. 6598--6603.

\bibitem{tang2019distributed}
W.~Tang and P.~Daoutidis, ``Distributed nonlinear model predictive control through accelerated parallel admm,'' in \emph{2019 American Control Conference (ACC)}.\hskip 1em plus 0.5em minus 0.4em\relax IEEE, 2019, pp. 1406--1411.

\bibitem{shorinwa2023distributed}
O.~Shorinwa and M.~Schwager, ``Distributed model predictive control via separable optimization in multi-agent networks,'' \emph{IEEE Transactions on Automatic Control}, 2023.

\bibitem{nedic2001distributed}
A.~Nedi{\'c}, D.~P. Bertsekas, and V.~S. Borkar, ``Distributed asynchronous incremental subgradient methods,'' \emph{Studies in Computational Mathematics}, vol.~8, no.~C, pp. 381--407, 2001.

\bibitem{nedic2009}
A.~Nedic and A.~Ozdaglar, ``Distributed subgradient methods for multi-agent optimization,'' \emph{IEEE Transactions on Automatic Control}, vol.~54, no.~1, pp. 48--61, 2009.

\bibitem{matei2011performance}
I.~Matei and J.~S. Baras, ``Performance evaluation of the consensus-based distributed subgradient method under random communication topologies,'' \emph{IEEE Journal of Selected Topics in Signal Processing}, vol.~5, no.~4, pp. 754--771, 2011.

\bibitem{olshevsky2009convergence}
A.~Olshevsky and J.~N. Tsitsiklis, ``Convergence speed in distributed consensus and averaging,'' \emph{SIAM Journal on Control and Optimization}, vol.~48, no.~1, pp. 33--55, 2009.

\bibitem{benezit2010weighted}
F.~B{\'e}n{\'e}zit, V.~Blondel, P.~Thiran, J.~Tsitsiklis, and M.~Vetterli, ``Weighted gossip: Distributed averaging using non-doubly stochastic matrices,'' in \emph{2010 IEEE International Symposium on Information Theory}.\hskip 1em plus 0.5em minus 0.4em\relax IEEE, 2010, pp. 1753--1757.

\bibitem{lobel2010distributed}
I.~Lobel and A.~Ozdaglar, ``Distributed subgradient methods for convex optimization over random networks,'' \emph{IEEE Transactions on Automatic Control}, vol.~56, no.~6, pp. 1291--1306, 2010.

\bibitem{yuan2016convergence}
K.~Yuan, Q.~Ling, and W.~Yin, ``On the convergence of decentralized gradient descent,'' \emph{SIAM Journal on Optimization}, vol.~26, no.~3, pp. 1835--1854, 2016.

\bibitem{shi2015extra}
W.~Shi, Q.~Ling, G.~Wu, and W.~Yin, ``{EXTRA}: An exact first-order algorithm for decentralized consensus optimization,'' \emph{SIAM Journal on Optimization}, vol.~25, no.~2, pp. 944--966, 2015.

\bibitem{qu2017harnessing}
G.~Qu and N.~Li, ``Harnessing smoothness to accelerate distributed optimization,'' \emph{IEEE Transactions on Control of Network Systems}, vol.~5, no.~3, pp. 1245--1260, 2017.

\bibitem{nedic2017achieving}
A.~Nedic, A.~Olshevsky, and W.~Shi, ``Achieving geometric convergence for distributed optimization over time-varying graphs,'' \emph{SIAM Journal on Optimization}, vol.~27, no.~4, pp. 2597--2633, 2017.

\bibitem{chen2012diffusion}
J.~Chen and A.~H. Sayed, ``Diffusion adaptation strategies for distributed optimization and learning over networks,'' \emph{IEEE Transactions on Signal Processing}, vol.~60, no.~8, pp. 4289--4305, 2012.

\bibitem{yuan2018exact}
K.~Yuan, B.~Ying, X.~Zhao, and A.~H. Sayed, ``Exact diffusion for distributed optimization and learning—part i: Algorithm development,'' \emph{IEEE Transactions on Signal Processing}, vol.~67, no.~3, pp. 708--723, 2018.

\bibitem{yuan2018exact2}
------, ``Exact diffusion for distributed optimization and learning—part {II}: Convergence analysis,'' \emph{IEEE Transactions on Signal Processing}, vol.~67, no.~3, pp. 724--739, 2018.

\bibitem{xu2015augmented}
J.~Xu, S.~Zhu, Y.~C. Soh, and L.~Xie, ``Augmented distributed gradient methods for multi-agent optimization under uncoordinated constant stepsizes,'' in \emph{2015 54th IEEE Conference on Decision and Control (CDC)}.\hskip 1em plus 0.5em minus 0.4em\relax IEEE, 2015, pp. 2055--2060.

\bibitem{saadatniaki2018optimization}
F.~Saadatniaki, R.~Xin, and U.~A. Khan, ``Optimization over time-varying directed graphs with row and column-stochastic matrices,'' \emph{arXiv preprint arXiv:1810.07393}, 2018.

\bibitem{zeng2017extrapush}
J.~Zeng and W.~Yin, ``{E}xtra{P}ush for convex smooth decentralized optimization over directed networks,'' \emph{Journal of Computational Mathematics}, vol.~35, no.~4, pp. 383--396, 2017.

\bibitem{xi2017add}
C.~Xi, R.~Xin, and U.~A. Khan, ``{ADD-OPT}: Accelerated distributed directed optimization,'' \emph{IEEE Transactions on Automatic Control}, vol.~63, no.~5, pp. 1329--1339, 2017.

\bibitem{xin2018linear}
R.~Xin and U.~A. Khan, ``A linear algorithm for optimization over directed graphs with geometric convergence,'' \emph{IEEE Control Systems Letters}, vol.~2, no.~3, pp. 315--320, 2018.

\bibitem{xin2019distributedNesterov}
R.~Xin, D.~Jakoveti{\'c}, and U.~A. Khan, ``Distributed {N}esterov gradient methods over arbitrary graphs,'' \emph{IEEE Signal Processing Letters}, vol.~26, no.~8, pp. 1247--1251, 2019.

\bibitem{qu2019accelerated}
G.~Qu and N.~Li, ``Accelerated distributed nesterov gradient descent,'' \emph{IEEE Transactions on Automatic Control}, 2019.

\bibitem{lu2020nesterov}
Q.~L{\"u}, X.~Liao, H.~Li, and T.~Huang, ``A {N}esterov-like gradient tracking algorithm for distributed optimization over directed networks,'' \emph{IEEE Transactions on Systems, Man, and Cybernetics: Systems}, 2020.

\bibitem{Mokhtari2015}
A.~Mokhtari, Q.~Ling, and A.~Ribeiro, ``{Network {N}ewton},'' \emph{Conference Record - Asilomar Conference on Signals, Systems and Computers}, vol. 2015-April, pp. 1621--1625, 2015.

\bibitem{Eisen2019}
M.~Eisen, A.~Mokhtari, and A.~Ribeiro, ``{A Primal-Dual Quasi-{N}ewton Method for Exact Consensus Optimization},'' \emph{IEEE Transactions on Signal Processing}, vol.~67, no.~23, pp. 5983--5997, 2019.

\bibitem{mansoori2019fast}
F.~Mansoori and E.~Wei, ``A fast distributed asynchronous newton-based optimization algorithm,'' \emph{IEEE Transactions on Automatic Control}, vol.~65, no.~7, pp. 2769--2784, 2019.

\bibitem{Ling2015}
Q.~Ling, W.~Shi, G.~Wu, and A.~Ribeiro, ``{DLM}: Decentralized linearized alternating direction method of multipliers,'' \emph{IEEE Transactions on Signal Processing}, vol.~63, no.~15, pp. 4051--4064, 2015.

\bibitem{chang2014multi}
T.-H. Chang, M.~Hong, and X.~Wang, ``Multi-agent distributed optimization via inexact consensus {ADMM},'' \emph{IEEE Transactions on Signal Processing}, vol.~63, no.~2, pp. 482--497, 2014.

\bibitem{farina2019distributed}
F.~Farina, A.~Garulli, A.~Giannitrapani, and G.~Notarstefano, ``A distributed asynchronous method of multipliers for constrained nonconvex optimization,'' \emph{Automatica}, vol. 103, pp. 243--253, 2019.

\bibitem{hestenes1952methods}
M.~R. Hestenes, E.~Stiefel \emph{et~al.}, ``Methods of conjugate gradients for solving linear systems,'' \emph{Journal of research of the National Bureau of Standards}, vol.~49, no.~6, pp. 409--436, 1952.

\bibitem{hager2006survey}
W.~W. Hager and H.~Zhang, ``A survey of nonlinear conjugate gradient methods,'' \emph{Pacific journal of Optimization}, vol.~2, no.~1, pp. 35--58, 2006.

\bibitem{dai1999nonlinear}
Y.-H. Dai and Y.~Yuan, ``A nonlinear conjugate gradient method with a strong global convergence property,'' \emph{SIAM Journal on optimization}, vol.~10, no.~1, pp. 177--182, 1999.

\bibitem{yuan2020conjugate}
G.~Yuan, T.~Li, and W.~Hu, ``A conjugate gradient algorithm for large-scale nonlinear equations and image restoration problems,'' \emph{Applied numerical mathematics}, vol. 147, pp. 129--141, 2020.

\bibitem{gilbert1992global}
J.~C. Gilbert and J.~Nocedal, ``Global convergence properties of conjugate gradient methods for optimization,'' \emph{SIAM Journal on optimization}, vol.~2, no.~1, pp. 21--42, 1992.

\bibitem{yuan2019global}
G.~Yuan, Z.~Wei, and Y.~Yang, ``The global convergence of the polak--ribi{\`e}re--polyak conjugate gradient algorithm under inexact line search for nonconvex functions,'' \emph{Journal of Computational and Applied Mathematics}, vol. 362, pp. 262--275, 2019.

\bibitem{shewchuk1994introduction}
J.~R. Shewchuk \emph{et~al.}, ``An introduction to the conjugate gradient method without the agonizing pain,'' 1994.

\bibitem{ismail2013implementation}
L.~Ismail and R.~Barua, ``Implementation and performance evaluation of a distributed conjugate gradient method in a cloud computing environment,'' \emph{Software: Practice and Experience}, vol.~43, no.~3, pp. 281--304, 2013.

\bibitem{chen2004implementing}
F.~Chen, K.~B. Theobald, and G.~R. Gao, ``Implementing parallel conjugate gradient on the earth multithreaded architecture,'' in \emph{2004 IEEE International Conference on Cluster Computing (IEEE Cat. No. 04EX935)}.\hskip 1em plus 0.5em minus 0.4em\relax IEEE, 2004, pp. 459--469.

\bibitem{lanucara1999conjugate}
P.~Lanucara and S.~Rovida, ``Conjugate-gradient algorithms: An mpi open-mp implementation on distributed shared memory systems,'' in \emph{First European Workshop on OpenMP}, 1999, pp. 76--78.

\bibitem{helfenstein2012parallel}
R.~Helfenstein and J.~Koko, ``Parallel preconditioned conjugate gradient algorithm on gpu,'' \emph{Journal of Computational and Applied Mathematics}, vol. 236, no.~15, pp. 3584--3590, 2012.

\bibitem{engelmann2021essentially}
A.~Engelmann and T.~Faulwasser, ``Essentially decentralized conjugate gradients,'' \emph{arXiv preprint arXiv:2102.12311}, 2021.

\bibitem{xu2016distributed}
S.~Xu, R.~C. De~Lamare, and H.~V. Poor, ``Distributed estimation over sensor networks based on distributed conjugate gradient strategies,'' \emph{IET Signal Processing}, vol.~10, no.~3, pp. 291--301, 2016.

\bibitem{ping2021dcg}
H.~Ping, Y.~Wang, and D.~Li, ``Dcg: Distributed conjugate gradient for efficient linear equations solving,'' \emph{arXiv preprint arXiv:2107.13814}, 2021.

\bibitem{xu2020distributed}
C.~Xu, J.~Zhu, Y.~Shang, and Q.~Wu, ``A distributed conjugate gradient online learning method over networks,'' \emph{Complexity}, vol. 2020, pp. 1--13, 2020.

\bibitem{xiao2007distributed}
L.~Xiao, S.~Boyd, and S.-J. Kim, ``Distributed average consensus with least-mean-square deviation,'' \emph{Journal of parallel and distributed computing}, vol.~67, no.~1, pp. 33--46, 2007.

\bibitem{sayed2014diffusion}
A.~H. Sayed, ``Diffusion adaptation over networks,'' in \emph{Academic Press Library in Signal Processing}.\hskip 1em plus 0.5em minus 0.4em\relax Elsevier, 2014, vol.~3, pp. 323--453.

\bibitem{fletcher1964function}
R.~Fletcher and C.~M. Reeves, ``Function minimization by conjugate gradients,'' \emph{The computer journal}, vol.~7, no.~2, pp. 149--154, 1964.

\bibitem{polak1969note}
E.~Polak and G.~Ribiere, ``Note sur la convergence de directions conjuge{\'e}s. rev. francaise informat,'' \emph{Recherche Opertionelle, 3e ann{\'e}e}, vol.~16, pp. 35--43, 1969.

\bibitem{polyak1969conjugate}
B.~T. Polyak, ``The conjugate gradient method in extremal problems,'' \emph{USSR Computational Mathematics and Mathematical Physics}, vol.~9, no.~4, pp. 94--112, 1969.

\bibitem{zhu2010discrete}
M.~Zhu and S.~Mart{\'\i}nez, ``Discrete-time dynamic average consensus,'' \emph{Automatica}, vol.~46, no.~2, pp. 322--329, 2010.

\bibitem{xin2020general}
R.~Xin, S.~Pu, A.~Nedi{\'c}, and U.~A. Khan, ``A general framework for decentralized optimization with first-order methods,'' \emph{Proceedings of the IEEE}, vol. 108, no.~11, pp. 1869--1889, 2020.

\bibitem{xin2019distributed}
R.~Xin and U.~A. Khan, ``Distributed heavy-ball: A generalization and acceleration of first-order methods with gradient tracking,'' \emph{IEEE Transactions on Automatic Control}, vol.~65, no.~6, pp. 2627--2633, 2019.

\end{thebibliography}

\end{document}